\documentclass[11pt,letterpaper,leqno]{amsart}

\title[Quadratic Invariants of Semiexplicit Integrators]{Preservation of Quadratic Invariants\\by Semiexplicit Symplectic Integrators\\for Non-separable Hamiltonian Systems}
\author{Tomoki Ohsawa}
\address{Department of Mathematical Sciences, The University of Texas at Dallas, 800 W Campbell Rd, Richardson, TX 75080-3021}
\email{tomoki@utdallas.edu}

\date{\today}

\keywords{Symplectic integrator, non-separable Hamiltonian, quadratic invariants}

\subjclass[2020]{37M15,65P10}

\numberwithin{equation}{section}

\usepackage{graphicx,amssymb,caption,paralist,epstopdf}

\usepackage[margin=1in, marginpar=.5in]{geometry}

\usepackage[numbers,sort&compress]{natbib}
\usepackage[colorlinks=false]{hyperref}

\usepackage[nameinlink]{cleveref}

\usepackage{tikz}
\usetikzlibrary{arrows,positioning}
\usepackage{tikz-cd}

\usepackage{pifont}
\newcommand{\cmark}{\ding{51}}
\newcommand{\xmark}{\ding{55}}

\theoremstyle{plain}
\newtheorem{theorem}{Theorem}
\newtheorem*{theorem*}{Theorem}

\newtheorem{lemma}[theorem]{Lemma}
\newtheorem{proposition}[theorem]{Proposition}

\theoremstyle{definition}

\newtheorem{example}[theorem]{Example}

\theoremstyle{remark}
\newtheorem{remark}[theorem]{Remark}

\def\od#1#2{\dfrac{d#1}{d#2}}

\def\dzero#1#2{\left.\od{}{#1} #2 \right|_{#1=0}}

\def\parentheses#1{{\left(#1\right)}}

\def\Span{\mathop{\mathrm{span}}\nolimits} 

\def\norm#1{{\left\|#1\right\|}}

\def\DS{\displaystyle}
\def\R{\mathbb{R}}

\def\N{\mathbb{N}}
\def\defeq{\mathrel{\mathop:}=}
\def\eqdef{=\mathrel{\mathop:}}
\def\setdef#1#2{{\left\{ #1 \ |\ #2 \right\}}}

\def\diag{\operatorname{diag}}

\def\SO{\mathsf{SO}}

\def\Sp{\mathsf{Sp}}

\def\sp{\mathfrak{sp}}

\def\sym{\mathsf{sym}}

\def\dt{\Delta t}
\DeclareMathOperator{\sgn}{sgn}

\begin{document}

\footskip=.6in

\begin{abstract}
  We prove that the recently developed semiexplicit symplectic integrators for non-separable Hamiltonian systems preserve any linear and quadratic invariants possessed by the Hamiltonian systems. This is in addition to being symmetric and symplectic as shown in our previous work; hence it shares the crucial structure-preserving properties with some of the well-known symplectic Runge--Kutta methods such as the Gauss--Legendre methods. The proof follows two steps: First we show how the extended Hamiltonian system proposed by Pihajoki inherits linear and quadratic invariants in the extended phase space from the original Hamiltonian system. Then we show that this inheritance in turn implies that our integrator preserves the original linear and quadratic invariants in the original phase space. We also analyze preservation/non-preservation of these invariants by Tao's extended Hamiltonian system and the extended phase space integrators of Pihajoki and Tao. The paper concludes with numerical demonstrations of our results using a simple test case and a system of point vortices.
\end{abstract}

\maketitle

\section{Introduction}
\subsection{Extended Phase Space Integrators}
\label{sec:extended_integrators}
Consider the initial value problem of the Hamiltonian system
\begin{equation}
  \label{eq:Ham}
  \dot{z} = \mathbb{J}\, DH(z)
  \quad\text{or}\quad
  \left\{
    \begin{aligned}
      \dot{q} &= D_{2}H(q, p), \\
      \dot{p} &= -D_{1}H(q, p)
    \end{aligned}
  \right.
\end{equation}
with Hamiltonian $H\colon T^{*}\R^{d} \to \R$ and the initial condition $z(0) = (q(0), p(0)) = (q_{0}, p_{0})$, where
\begin{equation*}
  \mathbb{J} \defeq
  \begin{bmatrix}
    0 & I_{d} \\
    -I_{d} & 0
  \end{bmatrix},
\end{equation*}
and $D$ stands for the Jacobian (gradient in this case) and $D_{i}$ stands for the partial derivative with respect to the $i^{\rm th}$ set of variables.

We would like to numerically solve the initial value problem efficiently and accurately.
For efficiency, one would prefer explicit methods, whereas for accuracy, one prefers to use those integrators that preserve the underlying geometric structures of the system~\eqref{eq:Ham}, such as the symplecticity of its flow and its invariants or conserved quantities.

It turns out that achieving both efficiency and accuracy in the above sense is quite challenging for general non-separable Hamiltonians, i.e., those $H(q,p)$ that cannot be written as $K(p)+ V(q)$ with some functions $K$ and $V$.
While there exist some explicit symplectic integrators for certain classes of non-separable Hamiltonian systems~\cite{StLa2000,Bl2002,WuFoRo2003,McQu2004,Ch2009,Ta2016a,WaSuLiWu2021a,WaSuLiWu2021b,WaSuLiWu2021c,WuWaSuLi2021,WuWaSuLiHa2022}, the choice of symplectic integrators for other non-separable systems has been mostly limited to symplectic (partitioned) Runge--Kutta methods, which are known to be implicit in general.

The recent development of extended phase space integrators is an attempt to change this landscape.
Specifically, instead of solving \eqref{eq:Ham} directly, \citet{Pi2015} proposed to solve
\begin{equation}
  \label{eq:Pihajoki}
  \begin{aligned}
    \dot{q} &= D_{2}H(x, p), \qquad & \dot{p} &= -D_{1}H(q, y), \medskip\\
    \dot{x} &= D_{2}H(q, y), \qquad & \dot{y} &= -D_{1}H(x, p)
  \end{aligned}
\end{equation}
with the initial condition
\begin{equation*}
  (q(0), x(0), p(0), y(0)) = (q_{0}, q_{0}, p_{0}, p_{0}).
\end{equation*}

Notice that \eqref{eq:Pihajoki} is a Hamiltonian system defined on the extended phase space
\begin{equation*}
  T^{*}\R^{2d} = \setdef{ (q, x, p, y) }{ (q, x) \in \R^{2d},\ (p, y) \in T_{(q,x)}^{*}\R^{2d} \cong \R^{2d} } \cong \R^{4d}
\end{equation*}
with the extended Hamiltonian
\begin{equation}
  \label{eq:hatH}
  \hat{H}\colon T^{*}\R^{2d} \to \R;
  \qquad
  \hat{H}(q, x, p, y) \defeq H(q, y) + H(x, p).
\end{equation}
Its solution satisfies $(q(t), p(t)) = (x(t), y(t))$ for any $t \in \R$ (assuming that the solution exists and is unique), and $t \mapsto (q(t),p(t))$ coincides with the solution of the initial value problem of the original Hamiltonian system~\eqref{eq:Ham}.
Geometrically speaking, the subspace
\begin{equation}
  \label{eq:mathcalN}
  \mathcal{N} \defeq \setdef{ (q, q, p, p) \in T^{*}\R^{2d} }{ (q, p) \in T^{*}\R^{d} } \subset T^{*}\R^{2d}
\end{equation}
is an invariant submanifold of \eqref{eq:Pihajoki}, and the system~\eqref{eq:Pihajoki} restricted to $\mathcal{N}$ gives two copies of the original system~\eqref{eq:Ham}.

Let us write
\begin{equation*}
  \zeta = (q, x, p, y),
  \qquad
  \eta = (q, y),
  \qquad
  \xi = (x, p).
\end{equation*}
We note in passing that, throughout this paper, vectors are usually column vectors, but we often write column vectors as tuples to save space just like we did above.
Then, we may write the extended Hamiltonian~\eqref{eq:hatH} as
\begin{equation*}
  \hat{H}(\zeta)
  = H(\eta) + H(\xi),
\end{equation*}
and then write \eqref{eq:Pihajoki} as follows:
\begin{equation}
  \label{eq:Pihajoki-xi_eta}
  \dot{\xi} = \mathbb{J}\,DH(\eta),
  \qquad
  \dot{\eta} = \mathbb{J}\,DH(\xi).
\end{equation}

This form is reminiscent of what happens to the original Hamiltonian system \eqref{eq:Ham} when $H$ is separable:
\begin{equation*}
  \dot{q} = DK(p),
  \qquad
  \dot{p} = -DV(q).
\end{equation*}
One can then show that the St\"ormer--Verlet integrator is actually a Strang splitting~\cite{St1968} consisting of the following two flows: 
\begin{equation*}
  \left\{
    \begin{aligned}
      \dot{q} &= 0, \\
      \dot{p} &= -DV(q)
    \end{aligned}
  \right.
  \quad\text{and}\quad
  \left\{
    \begin{aligned}
      \dot{q} &= DK(p), \\
      \dot{p} &= 0.
    \end{aligned}
  \right.
\end{equation*}
\citet{Pi2015} proposed to do the same with \eqref{eq:Pihajoki-xi_eta}:
Let $\hat{\Phi}^{A}, \hat{\Phi}^{B}$ be the flows of
\begin{equation*}
  \left\{
    \begin{aligned}
      \dot{\eta} &= 0, \\
      \dot{\xi} &= \mathbb{J}\,DH(\eta)
    \end{aligned}
  \right.
  \quad\text{and}\quad
  \left\{
    \begin{aligned}
      \dot{\eta} &= \mathbb{J}\,DH(\xi), \\
      \dot{\xi} &= 0,
    \end{aligned}
  \right.
\end{equation*}
respectively, that is,
\begin{equation}
\label{eq:PhiAB}
  \hat{\Phi}^{A}_{t}\colon
  (\eta, \xi) \mapsto
  \parentheses{
    \eta,
    \xi + t\,\mathbb{J}\,DH(\eta)
  }
  \quad\text{and}\quad
  \hat{\Phi}^{B}_{t}\colon
  (\eta, \xi) \mapsto
  \parentheses{
    \eta + t\,\mathbb{J}\,DH(\xi),
    \xi
  }.
\end{equation}
Then the Strang splitting
\begin{equation}
  \label{eq:Pihajoki-Strang}
  \hat{\Phi}_{\dt} \defeq \hat{\Phi}^{A}_{\dt/2} \circ \hat{\Phi}^{B}_{\dt} \circ \hat{\Phi}^{A}_{\dt/2}
\end{equation}
gives a $2^{\rm nd}$-order explicit integrator with time step $\dt$ for the extended Hamiltonian system~\eqref{eq:Pihajoki}.

\subsection{Semiexplicit Integrator with Symmetric Projection}
\label{ssec:semiexplicit}
Unfortunately, Pihajoki's integrator~\eqref{eq:Pihajoki-Strang} has some issues: (i)~the numerical solution does not stay in the subspace $\mathcal{N}$, and even worse, the defect $(x - q, y - p)$ in the phase space copies $(q,p)$ and $(x,y)$ tends to grow in time numerically; (ii)~the method is symplectic in the \textit{extended} phase space $T^{*}\R^{2d}$ but not in the \textit{original} phase space $T^{*}\R^{d}$.

Various modifications of the extended phase space integrator have been proposed to mitigate the first issue, most notably by \citet{Ta2016b} (see also \Cref{sec:Tao}); see also \cite{LiWu2017,LiWuHuLi2016,LuWuHuLi2017,PhysRevD.104.044055,ZhZhLiWu2022} for relativistic dynamics with astrophysical applications.
However, none of them fully resolves both issues.

In the recent work~\cite{JaOh2023}, we proposed to address the first issue using the symmetric projection~(see, e.g., \citet[Section~V.4.1]{HaLuWa2006}) to the subspace $\mathcal{N}$:
First notice that the subspace $\mathcal{N}$ defined in \eqref{eq:mathcalN} is written as
\begin{equation}
  \label{eq:A}
  \mathcal{N} = \ker A
  \quad\text{with}\quad
  A \defeq
  \begin{bmatrix}
    I_{d} & -I_{d} & 0 & 0 \\
    0 & 0 & I_{d} & -I_{d}
  \end{bmatrix}.
\end{equation}
Then, using Pihajoki's extended phase space integrator $\hat{\Phi}_{\dt}$ from \eqref{eq:Pihajoki-Strang}, we defined our semiexplicit integrator as follows (see also \Cref{fig:scheme} below):
Given $z_{n} = (q_{n}, p_{n}) \in T^{*}\R^{d}$,
\begin{enumerate}
  \renewcommand{\theenumi}{\arabic{enumi}}
  \renewcommand{\labelenumi}{\sf\theenumi.}
\item $\zeta_{n} \defeq (q_{n}, q_{n}, p_{n}, p_{n})$;
\item Find $\mu \in \R^{2d}$ such that $\hat{\Phi}_{\dt}(\zeta_{n} + A^{T}\mu) + A^{T} \mu \in \mathcal{N}$;
  \label{step:shift1}
\item $\hat{\zeta}_{n} \defeq \zeta_{n} + A^{T} \mu$;
\item $\hat{\zeta}_{n+1} \defeq \hat{\Phi}_{\dt}(\hat{\zeta}_{n})$;
\item $\zeta_{n+1} = (q_{n+1}, q_{n+1}, p_{n+1}, p_{n+1}) \defeq \hat{\zeta}_{n+1} + A^{T} \mu$;
  \label{step:shift2}
\item $z_{n+1} \defeq (q_{n+1}, p_{n+1})$.
\end{enumerate}
Note that Steps~\ref{step:shift1}--\ref{step:shift2} combined are equivalent to solving the nonlinear equations
\begin{equation*}
  F_{\dt}(\zeta_{n+1}, \mu) \defeq
  \begin{bmatrix}
    \zeta_{n+1} - \hat{\Phi}_{\dt}(\zeta_{n} + A^{T}\mu) - A^{T}\mu \\
    A \zeta_{n+1}
  \end{bmatrix}
  = 0
\end{equation*}
for $(\zeta_{n+1}, \mu) \in \R^{2d} \times T^{*}\R^{2d}$, or eliminating $\zeta_{n+1}$, the following nonlinear equation for $\mu$:
\begin{equation*}
  f_{\dt}(\mu) \defeq A\parentheses{ \hat{\Phi}_{\dt}(\zeta_{n} + A^{T}\mu) + A^{T}\mu } = 0.
\end{equation*}
  
\begin{figure}[hbtp]
  \centering
  \includegraphics[width=.6\linewidth]{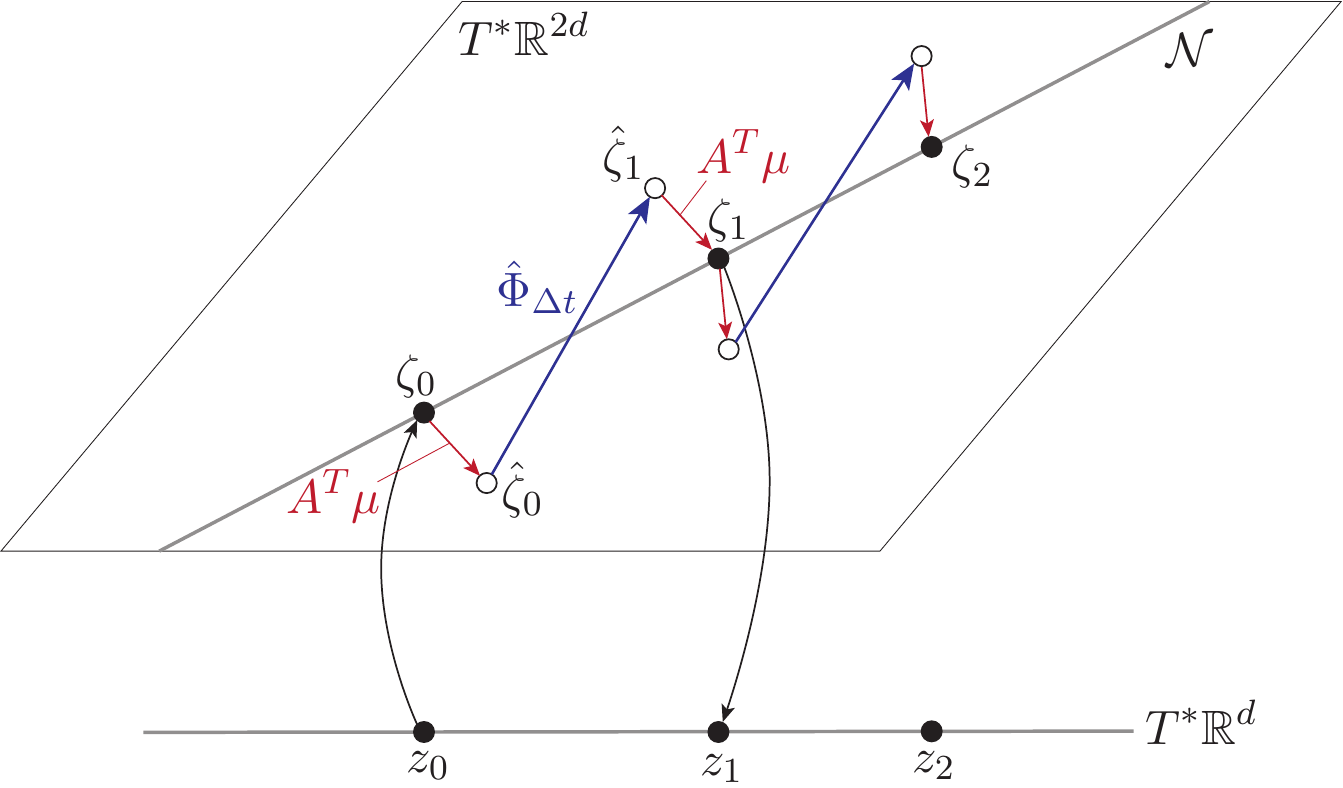}
  \caption{Extended phase space integrator with symmetric projection~\cite{JaOh2023}.}
  \label{fig:scheme}
\end{figure}

One may construct a higher-order method by replacing $\hat{\Phi}$ by a higher-order composition of the 2nd-order method~\eqref{eq:Pihajoki-Strang}, such as the the Triple Jump, Suzuki's, and Yoshida's compositions~\cite{CrGo1989,Forest1989,Su1990,Yo1990}; see also our previous work \cite[Section~4.1]{JaOh2023} for details.

It turns out that the above semiexplicit integrator not only eliminates the defect $(x - q, y - p)$, but also is symplectic in the \textit{original} phase space $T^{*}\R^{d}$~\cite{JaOh2023}, hence resolving the second issue mentioned above as well.
Additionally, it is also symmetric by construction.
Moreover, by using a simplified Newton's method or the quasi-Newton method of \citet{broyden1965class} and a small enough time step $\dt$, the implicit step of solving nonlinear equation tends to be fast; as a result, our method is comparable in speed to and sometimes faster than the fully explicit method of \citet{Ta2016b} and the symplectic Runge--Kutta methods, especially for higher-order implementations; see \Cref{sec:efficiency} for numerical results on efficiency.

\subsection{Main Result}
This paper addresses the preservation of linear and quadratic invariants by our semiexplicit integrator---yet another desired property for structure-preserving integrators in addition to symmetry and symplecticity.

It is well known that the St\"ormer--Verlet method is a special case of the partitioned Runge--Kutta method with the 2-stage Lobatto IIIA--IIIB pair applied to a separable Hamiltonian; see, e.g., \citet[Section~8.5.3]{SaCa2018}, \citet[Section~6.3.2]{LeRe2004}, \citet[Section~II.2.1]{HaLuWa2006}, and also \citet{Ge1993}.
Such methods applied to \eqref{eq:Ham} are known to preserve linear and quadratic invariants of the form $a^{T}z$ and $q^{T} W p$, respectively, with $a \in \R^{2d}$ and $W \in \R^{d \times d}$; see, e.g., \cite[Section~6.3.2]{LeRe2004} and \cite[Theorems~IV.2.3]{HaLuWa2006}.

Given that the time evolution part of our method is the extended-phase-space analogue~\eqref{eq:Pihajoki-Strang} of the St\"ormer--Verlet method, one may expect that the best we can hope for with our integrator would be to preserve quadratic invariants of the form $q^{T} W p$ but not a general quadratic invariants of the form $z^{T} \Sigma z$ with symmetric $\Sigma \in \R^{2d\times 2d}$.
Such a limitation is not desirable for an integrator for non-separable Hamiltonian systems because they often possess invariants of the form $q^{T} M q + p^{T} N p$ with symmetric $M,N \in \R^{d\times d}$.

Our main result is that our integrator preserves any linear and quadratic invariants of the original Hamiltonian system~\eqref{eq:Ham} without such a limitation:
\begin{theorem}
  \label{thm:main}
  The semiexplicit integrator~\cite{JaOh2023} defined in \Cref{ssec:semiexplicit} preserves any linear and quadratic invariants of the original Hamiltonian system~\eqref{eq:Ham}.
\end{theorem}

\begin{remark}
  The same statement holds for any higher-order semiexplicit method constructed by replacing the $2^{\rm nd}$-order integrator \eqref{eq:Pihajoki-Strang} by its higher-order variant using the Triple Jump composition (see \cite{CrGo1989,Forest1989,Su1990,Yo1990} and \cite[Example~II.4.2]{HaLuWa2006}) or those of \citet{Su1990} and \citet{Yo1990} (see also \cite[Example~II.4.3, Section~V.3.2]{HaLuWa2006}).
  These higher order integrators are tested in \cite{JaOh2023} as well.
\end{remark}

\begin{remark}
  As we shall discuss later, neither Pihajoki's nor Tao's~\cite{Ta2016b} integrator has the property described in the above theorem for quadratic invariants in a strict sense.
\end{remark}

To our knowledge, the only integrators for general non-separable Hamiltonian systems that are symmetric, symplectic in the original phase space $T^{*}\R^{d}$, and preserve any linear and quadratic invariants are symplectic Runge--Kutta methods, such as the Gauss--Legendre methods; see \citet{Co1987} and also \cite[Section~6.3.1]{LeRe2004} and \cite[Theorems~IV.2.1 and IV.2.2]{HaLuWa2006}.

\subsection{Outline}
\label{ssec:outline}
We shall show \Cref{thm:main} in the rest of the paper.
\Cref{fig:overview} provides an overview of our argument for quadratic invariants; a similar picture applies to linear invariants.

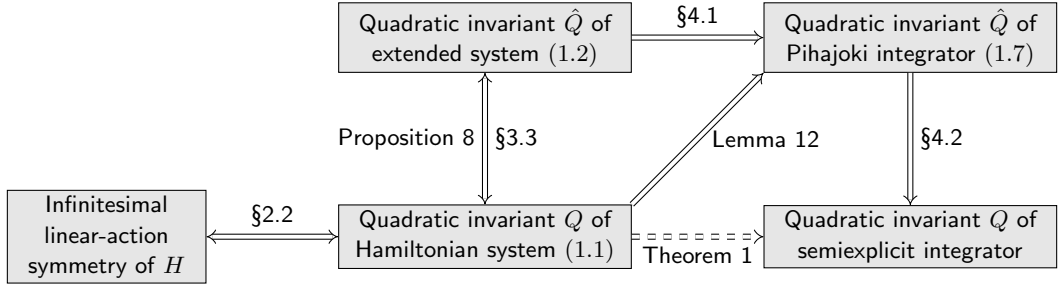
\begin{figure}[htbp]
  \centering
  \begin{tikzpicture}[node distance = 1.75cm, auto, font=\footnotesize\sffamily,
    block/.style={rectangle, draw, fill=black!10, inner sep=2pt, text width=3.75cm, text badly centered, minimum height=0.7cm, font=\footnotesize\sffamily}]
    %
    \node[block, text width=2.5cm] (g-sym) {Infinitesimal linear-action symmetry of $H$};
    \node[block] (quad) [right=of g-sym] {Quadratic invariant $Q$ of Hamiltonian system~\eqref{eq:Ham}};
    \node[block] (main) [right=of quad] {Quadratic invariant $Q$ of semiexplicit integrator};
    \node[block] (quad-ext) [above=of quad] {Quadratic invariant $\hat{Q}$ of extended system~\eqref{eq:Pihajoki}};
    \node[block] (quad-P) [right=of quad-ext] {Quadratic invariant $\hat{Q}$ of Pihajoki integrator~\eqref{eq:Pihajoki-Strang}};
    %
    \draw[double distance=1.5pt, {Implies[]}-{Implies[]}]
    (g-sym) -- (quad) node[above,midway] {\S\ref{ssec:quad_inv}};
    \draw[double distance=1.5pt, {Implies[]}-{Implies[]}]
    (quad) -- (quad-ext) node[left,midway] {\Cref{prop:inheritance}} node[right,midway] {\S\ref{ssec:quad_inv-ext}};
    \draw[double distance=1.5pt, -{Implies[]}]
    (quad-ext) -- (quad-P) node[above,midway] {\S\ref{ssec:Pihajoki}};
    \draw[double distance=1.5pt, -{Implies[]}]
    (quad-P) -- (main) node[right,midway] {\S\ref{ssec:proof}};
    \draw[dashed,double distance=1.5pt, -{Implies[]}]
    (quad) -- (main) node[below,midway] {\Cref{thm:main}};
    \draw[double distance=1.5pt, -{Implies[]}]
    (quad.north east) -- (quad-P.south west) node[right,midway] {\,\Cref{lem:Pihajoki}};
  \end{tikzpicture}
  \caption{Overview of our results on quadratic invariants. A similar picture applies to linear invariants although we do not discuss the corresponding symmetry.}
  \label{fig:overview}
\end{figure}

Since the main focus is on quadratic invariants, we first give, in \Cref{sec:symmetry}, a review of the relationship between the symmetry by linear actions and quadratic invariants of the original Hamiltonian system~\eqref{eq:Ham}.
In \Cref{sec:invariants-ext}, we show how such symmetry and linear and quadratic invariants are inherited by the extended Hamiltonian system~\eqref{eq:Pihajoki}.
In \Cref{sec:proof}, we give a proof of \Cref{thm:main} after discussing a conservation law of Pihajoki's integrator~\eqref{eq:Pihajoki-Strang} as a key lemma.
Finally, in \Cref{sec:numerical_results}, we first discuss and summarize conservation and non-conservation of these invariants for those extended phase space integrators of Pihajoki, Tao, and ours.
We then test these three integrators numerically to demonstrate these properties including \Cref{thm:main}.

\section{Symmetry and Quadratic Invariants of Hamiltonian Systems}
\label{sec:symmetry}
This section gives a review of symmetry and conservation laws in Hamiltonian systems focusing on linear and quadratic invariants.
We do not discuss symmetry behind linear invariants here, because it is more straightforward to focus on the linear invariants themselves without delving into the (translational) symmetries behind them.

However, we shall discuss in detail the relationship between the symmetry under linear actions and quadratic invariants, because the underlying algebraic structure gives a better idea of how quadratic invariants are inherited by the extended system, as we shall see in \Cref{sec:invariants-ext}.
The upshot is that an infinitesimal symmetry under a linear action implies a quadratic invariant, and vice versa.

\subsection{Symmetry under Linear Action}
Let us define the symplectic group
\begin{equation*}
  \Sp(2d,\R) \defeq \setdef{ S \in \R^{2d\times2d} }{ S^{T} \mathbb{J} S = \mathbb{J} }
\end{equation*}
or equivalently, written as block matrices consisting of $d \times d$ submatrices,
\begin{equation}
  \label{def:Sp2dR-block}
  \Sp(2d,\R) \defeq
  \setdef{
    \begin{bmatrix}
      \mathsf{A} & \mathsf{B} \\
      \mathsf{C} & \mathsf{D}
    \end{bmatrix}
    \in \R^{2d\times2d}
  }{
    \mathsf{A}^{T}\mathsf{C} = \mathsf{C}^{T}\mathsf{A},\,
    \mathsf{B}^{T}\mathsf{D} = \mathsf{D}^{T}\mathsf{B},\,
    \mathsf{A}^{T}\mathsf{D} - \mathsf{C}^{T}\mathsf{B} = I_{d}
  }.
\end{equation}

Let $\mathsf{G}$ be a matrix Lie subgroup of $\Sp(2d,\R)$, and consider the standard action of $\mathsf{G}$ on $T^{*}\R^{d}$ by matrix-vector multiplication:
\begin{equation*}
  \Psi\colon \mathsf{G} \times T^{*}\R^{d} \to T^{*}\R^{d};
  \qquad
  \parentheses{ S, z =
    \begin{bmatrix}
      q \\
      p
    \end{bmatrix}
  } \mapsto S z \eqdef \Psi_{S}(z).
\end{equation*}
Then the action is symplectic because
\begin{equation*}
  (D\Psi_{S})^{T}\, \mathbb{J}\, D\Psi_{S} = \mathbb{J}
  \iff
  S^{T} \mathbb{J} S = \mathbb{J}
  \iff
  S \in \Sp(2d,\R),
\end{equation*}
where $D$ denotes the Jacobian.

The Hamiltonian $H\colon T^{*}\R^{d} \to \R$ of the original system~\eqref{eq:Ham} is said to have $\mathsf{G}$-symmetry if
\begin{equation}
  \label{eq:G-symmetry}
  H \circ \Psi_{S} = H
  \iff
  H(S z) = H(z)
  \quad
  \forall
  S \in \mathsf{G}
  \quad
  \forall z \in T^{*}\R^{d}.
\end{equation}
We also say that $\mathsf{G}$ is a \textit{symmetry group} of the Hamiltonian $H$ or the Hamiltonian system~\eqref{eq:Ham} if the above is satisfied.

\begin{example}[Finite-dimensional NLS]
  \label{ex:NLS}
  As a finite-dimensional approximation to the nonlinear Schr\"odinger equation (NLS), \citet{CoKeStTaTa2010} (see also \citet{Ta2016b}) gave the Hamiltonian system~\eqref{eq:Ham} with $d = N$ and the following non-separable Hamiltonian:
  \begin{equation}
    \label{eq:H-NLS}
    H(q,p) = \frac{1}{4} \sum_{i=1}^{N} \bigl(q_i^2 + p_i^2\bigr)^2 - \sum_{i=2}^{N} \bigl( p_{i-1}^2 p_{i}^2 + q_{i-1}^2 q_{i}^2 - q_{i-1}^2 p_{i}^2 - p_{i-1}^2 q_{i}^2 + 4p_{i-1}p_{i}q_{i-1}q_{i} \bigr).
  \end{equation}
  Consider the subgroup
  \begin{equation*}
    \mathsf{G} \defeq \setdef{
      \begin{bmatrix}
        (\cos\theta) I_{d} & -(\sin\theta) I_{d} \\
        (\sin\theta) I_{d} & (\cos\theta) I_{d}
      \end{bmatrix}
      \in \R^{2d\times2d}
    }{ \theta \in \R }.
    \end{equation*}
  It is essentially $\SO(2)$, and in fact defines a homomorphism from $\SO(2)$ to $\Sp(2d,\R)$ and hence a subgroup $\mathsf{G}$ of $\Sp(2d,\R)$.
  Then we see that the Hamiltonian~\eqref{eq:H-NLS} possesses $\mathsf{G}$-symmetry in the sense that \eqref{eq:G-symmetry} holds.
  One may certainly take $\mathsf{G} = \SO(2)$ and define a group action $\Psi$ accordingly, but we would rather like to have $\mathsf{G}$ as a subgroup of $\Sp(2d,\R)$ because it gives a unified approach on quadratic invariants as we shall see in a moment.
\end{example}

\subsection{Infinitesimal Symmetry and Quadratic Invariants}
\label{ssec:quad_inv}
Let $\sp(2d,\R)$ be the Lie algebra of $\Sp(2d,\R)$, i.e.,
\begin{equation*}
  \sp(2d,\R) \defeq \setdef{
    \varkappa \in \R^{2d\times2d}
  }{
    \varkappa^{T}\mathbb{J} + \mathbb{J} \varkappa = 0
  }.
\end{equation*}
Instead of working directly with elements in $\sp(2d,\R)$, it is often more convenient to work with the space
\begin{equation*}
  \sym(2d,\R) \defeq \setdef{ \kappa \in \R^{2d\times2d} }{ \kappa^{T} = \kappa }
\end{equation*}
of real symmetric $2d \times 2d$ matrices via the following identification:
\begin{equation}
  \label{eq:kappas}
  \sym(2d,\R) \leftrightarrow \sp(2d,\R);
  \qquad
  \kappa =
  \begin{bmatrix}
    \kappa_{11} & \kappa_{12} \\
    \kappa_{12}^{T} & \kappa_{22}
  \end{bmatrix}
  = \mathbb{J}^{T} \varkappa
  \leftrightarrow
  \varkappa
  = \mathbb{J} \kappa
  =
  \begin{bmatrix}
    \kappa_{12}^{T} & \kappa_{22} \\
    -\kappa_{11} & -\kappa_{12}
  \end{bmatrix},
\end{equation}
where $\kappa_{12}$ is a $d \times d$ real (not necessarily symmetric) matrix, and $\kappa_{11}, \kappa_{22} \in \sym(d,\R)$.

Now, let $\mathfrak{g}$ be the Lie algebra of the symmetry group $\mathsf{G} \subset \Sp(2d,\R)$ of the Hamiltonian $H$.
Then $\mathfrak{g}$ is a subalgebra of $\sp(2d,\R)$, which can be identified with the subspace
\begin{equation}
  \label{eq:g_sym}
  \mathfrak{g}_{\sym} \defeq \mathbb{J}^{T} \mathfrak{g}
  = \setdef{ \kappa \defeq \mathbb{J}^{T} \varkappa \in \sym(2d,\R) }{ \varkappa \in \mathfrak{g} }
  \subset
  \sym(2d,\R).
\end{equation}
Then, for any $\varkappa = \mathbb{J} \kappa \in \sp(2d,\R)$, we may define a vector field called the \textit{infinitesimal generator} as follows:
\begin{equation}
  \label{eq:inf_gen}
  \kappa_{T^{*}\R^{d}}(z)
  \defeq \dzero{s}{ \exp(s \varkappa) z }
  = \varkappa z
  = \mathbb{J} \kappa z
  = \begin{bmatrix}
    \kappa_{12}^{T} q + \kappa_{22} p \\
    -\kappa_{11} q - \kappa_{12} p
  \end{bmatrix}.
 \end{equation}
Intuitively, this gives the infinitesimal symmetry direction of the Hamiltonian $H$.
Indeed, since $\exp(s \varkappa) \in \mathsf{G}$ for any $s \in \R$ and any $\varkappa \in \mathfrak{g}$, \eqref{eq:G-symmetry} implies $H( \exp(s \varkappa) z) = H(z)$, and taking the derivative of both sides with respect to $s$ at $s = 0$, we have
\begin{equation}
  \label{eq:g-symmetry}
  \kappa_{T^{*}\R^{d}}(z)^{T} DH(z) = 0
  \quad
  \forall
  \kappa \in \mathfrak{g}_{\sym}
  \quad
  \forall z \in T^{*}\R^{d},
\end{equation}
showing the infinitesimal invariance of $H$ in the directions defined by $\kappa_{T^{*}\R^{d}}$.
Hence we shall refer to it as an \textit{infinitesimal symmetry} (or $\mathfrak{g}$-symmetry) of $H$.
This is what the lower left box in \Cref{fig:overview} signifies.

What is the associated Noether invariant?
For any $\kappa \in \mathfrak{g}_{\sym}$, define
\begin{equation}
  \label{eq:Q}
  Q_{\kappa}(z) \defeq \frac{1}{2} z^{T} \kappa z
  = \frac{1}{2} q^{T} \kappa_{11} q + q^{T} \kappa_{12} p + \frac{1}{2} p^{T} \kappa_{22} p
\end{equation}
so that one has
\begin{equation*}
  \kappa_{T^{*}\R^{d}}(z) = \mathbb{J}\,DQ_{\kappa}(z)
  \quad
  \forall
  z \in T^{*}\R^{d}.
\end{equation*}
Then, taking the Poisson bracket of $Q_{\kappa}$ and $H$, 
\begin{align*}
  \{Q_{\kappa}, H\}(z)
  &= DQ_{\kappa}(z)^{T} \mathbb{J} DH(z) \\
  &= -(\mathbb{J}\,DQ_{\kappa}(z))^{T} DH(z) \\
  &= -\kappa_{T^{*}\R^{d}}(z)^{T} DH(z).
\end{align*}
Therefore, the infinitesimal symmetry~\eqref{eq:g-symmetry} implies that $Q_{\kappa}$ is an invariant of Hamiltonian system \eqref{eq:Ham} for any $\kappa \in \mathfrak{g}_{\sym}$.

Conversely, suppose that \eqref{eq:Ham} possesses a quadratic invariant.
One can find $\kappa \in \sym(2d,\R)$ so that the invariant is written as $Q_{\kappa}$ as in \eqref{eq:Q}.
Then $\{Q_{\kappa}, H\}(z) = 0$ for any $z \in T^{*}\R^{d}$, and thus the above equality implies the infinitesimal symmetry~\eqref{eq:g-symmetry} for that particular $\kappa$.

For a family of quadratic invariants $\{ \mathcal{I}_{i} \}_{i=1}^{k}$, one may find $\{ \kappa_{i} \}_{i=1}^{k} \subset \sym(2d,\R)$ so that $Q_{\kappa_{i}} = \mathcal{I}_{i}$ for $i \in \{1, \dots, k\}$.
Then, setting $\mathfrak{g}_{\sym} = \Span\{ \kappa_{i} \}_{i=1}^{k}$, the corresponding $\mathfrak{g} = \mathbb{J} \mathfrak{g}_{\sym} \subset \sp(2d,\R)$ gives the symmetry Lie algebra, i.e., the Hamiltonian $H$ satisfies $\mathfrak{g}$-symmetry~\eqref{eq:g-symmetry}.

\begin{example}
  \label{ex:NLS2}
  Consider again the NLS from \Cref{ex:NLS}.
  The Lie algebra $\mathfrak{g}$ here is
  \begin{equation*}
    \mathfrak{g} = \Span\left\{
      \varkappa_{0} \defeq 
      \begin{bmatrix}
        0 & -I_{d} \\
        I_{d} & 0
      \end{bmatrix}
    \right\},
  \end{equation*}
  and thus
  \begin{equation*}
    \mathfrak{g}_{\sym} = \Span\left\{
      \mathbb{J}^{T} \varkappa_{0} = -I_{2d}
    \right\}.
  \end{equation*}
  Therefore, setting $\kappa = 2I_{2d}$, the associated quadratic invariant is
  \begin{equation}
    \label{eq:Q-NLS}
    Q_{\kappa}(z) = \frac{1}{2} z^{T} \kappa z
    = \sum_{i=1}^{d} (q_{i}^{2} + p_{i}^{2}),
  \end{equation}
  which is essentially the ``total mass'' of the NLS~\cite{CoKeStTaTa2010}.
\end{example}

\section{Linear and Quadratic Invariants in Extended System}
\label{sec:invariants-ext}
Now we would like to address the following question:
If a Hamiltonian system~\eqref{eq:Ham} possesses linear and/or quadratic invariants, then does the corresponding extended system~\eqref{eq:Pihajoki} inherit such invariants?

We first discuss linear invariants in \Cref{ssec:linear_inv}, and then in \Cref{ssec:action-ext,ssec:quad_inv-ext}, we build on the previous section to discuss how linear action symmetries and quadratic invariants are inherited by the extended system~\eqref{eq:Pihajoki}.

\subsection{Linear Invariants of Extended System}
\label{ssec:linear_inv}
\begin{proposition}[Inheritance of linear invariants]
  \label{prop:inheritance-linear}
  The function
  \begin{equation}
    \label{eq:L}
    L_{a}(z) \defeq a^{T} z
    \quad\text{with}\quad
    a = (a_{q}, a_{p}) \in \R^{2d}
    \quad\text{and}\quad
    a_{q}, a_{p} \in \R^{d}
  \end{equation}
  is a linear invariant of the original Hamiltonian system~\eqref{eq:Ham} if and only if
  \begin{equation}
    \label{eq:hatL}
    \hat{L}_{a}(\zeta) \defeq \hat{a}^{T} \zeta
    \quad\text{with}\quad
    \zeta = (q, x, p, y) \in T^{*}\R^{2d} \cong \R^{4d}
    \quad\text{and}\quad
    \hat{a} = \frac{1}{2}(a_{q}, a_{q}, a_{p}, a_{p}) \in \R^{4d}
  \end{equation}
  is a linear invariant of the extended Hamiltonian system~\eqref{eq:Pihajoki}.
\end{proposition}
\begin{proof}[Proof of \Cref{prop:inheritance-linear}]
  Notice that
  \begin{equation}
    \label{eq:R-symmetry}
    \begin{split}
      \text{\eqref{eq:L} is an invariant of \eqref{eq:Ham}}
      &\iff
      \{L_{a}, H\}(z) = 0\quad \forall z \in T^{*}\R^{d} \\
      &\iff a^{T} \mathbb{J}\, DH(z) = 0 \quad \forall z \in T^{*}\R^{d}.
    \end{split}
  \end{equation}
  
  On the other hand, for the extended system, let us define, using the Kronecker product $\otimes$,
  \begin{equation*}
    \hat{\mathbb{J}}
    \defeq
    \begin{bmatrix}
      0 & I_{2d} \\
      -I_{2d} & 0 
    \end{bmatrix}
    = 
    \begin{bmatrix}
      0 & I_{2} \otimes I_{d} \\
      -I_{2} \otimes I_{d} & 0
    \end{bmatrix}
  \end{equation*}
  and also the extended Poisson bracket
  \begin{equation*}
    \{\hat{F}, \hat{G}\}_{\text{ext}}(\zeta) \defeq \bigl( D\hat{F}(\zeta) \bigr)^{T} \hat{\mathbb{J}}\, D\hat{G}(\zeta).
  \end{equation*}
  Then we have
  \begin{equation}
    \label{eq:R-symmetry-ext}        
    \begin{split}
      &\text{\eqref{eq:hatL} is an invariant of \eqref{eq:Pihajoki}}
      \\
      &\iff
      \{\hat{L}_{a}, \hat{H}\}_{\text{ext}}(\zeta) = 0 \quad \forall \zeta \in T^{*}\R^{2d}
      \\
      &\iff
      \hat{a}^{T} \hat{\mathbb{J}}\,D\hat{H}(\zeta) = 0
      \quad \forall \zeta \in T^{*}\R^{2d}
      \\
      &\iff
      \frac{1}{2} \parentheses{
        a^{T} \mathbb{J}\, DH(q,y)
        + a^{T} \mathbb{J}\, DH(x,p)
      } = 0
      \quad \forall (q,y), (x,p) \in T^{*}\R^{d}.
    \end{split}
  \end{equation}
  Clearly \eqref{eq:R-symmetry} implies \eqref{eq:R-symmetry-ext}.
  On the other hand, \eqref{eq:R-symmetry-ext} for the particular case of $(x,p) = (q,y)$ gives
  \begin{equation*}
    a^{T} \mathbb{J}\, DH(q,y) = 0
    \quad \forall (q,y) \in T^{*}\R^{d},
  \end{equation*}
  which implies \eqref{eq:R-symmetry}.
  Hence the claimed equivalence follows.
\end{proof}

\begin{remark}
  The extended system of \citet{Ta2016b} (see \eqref{eq:Tao} in \Cref{sec:Tao}) enjoys the same property; see \Cref{ssec:inheritance-Tao}.
\end{remark}

\subsection{Actions on Extended Phase Space}
\label{ssec:action-ext}
Just as in the last section, let $\mathsf{G}$ be a subgroup of $\Sp(2d,\R)$, and consider the following action of $\mathsf{G}$ on the extended phase space $T^{*}\R^{2d} \cong \R^{4d}$:
\begin{equation}
  \label{eq:hatPsi}
  \begin{split}
    \hat{\Psi}&\colon \mathsf{G} \times T^{*}\R^{2d} \to T^{*}\R^{2d};
    \\\
    &\parentheses{
      S = \begin{bmatrix}
        \mathsf{A} & \mathsf{B} \\
        \mathsf{C} & \mathsf{D}
      \end{bmatrix},
      \zeta =
      \begin{bmatrix}
        q \\
        x \\
        p \\
        y
      \end{bmatrix}
    }
    \mapsto
    \begin{bmatrix}
      \mathsf{A} q + \mathsf{B} y \\
      \mathsf{A} x + \mathsf{B} p \\
      \mathsf{C} x + \mathsf{D} p \\
      \mathsf{C} q + \mathsf{D} y
    \end{bmatrix}
    = \hat{S} \zeta
    \eqdef \hat{\Psi}_{S}\parentheses{ \zeta },
  \end{split}
\end{equation}
where we defined, again using the Kronecker product $\otimes$,
\begin{equation*}
  \hat{S} \defeq
  \begin{bmatrix}
    \mathsf{A} & 0 & 0 & \mathsf{B} \\
    0 & \mathsf{A} & \mathsf{B} & 0 \\
    0 & \mathsf{C} & \mathsf{D} & 0 \\
    \mathsf{C} & 0 & 0 & \mathsf{D}
  \end{bmatrix}
  =
  \begin{bmatrix}
    I_{2} \otimes \mathsf{A} & P \otimes \mathsf{B} \\
    P \otimes \mathsf{C} & I_{2} \otimes \mathsf{D}
  \end{bmatrix}
  \quad
  \text{with}
  \quad
  P \defeq
  \begin{bmatrix}
    0 & 1 \\
    1 & 0
  \end{bmatrix}.
\end{equation*}
Then it is a straightforward computation to show that $\hat{S}^{T} \hat{\mathbb{J}} \hat{S} = \hat{\mathbb{J}}$, i.e., $\hat{S} \in \Sp(4d,\R)$.

Likewise, for any $\varkappa = \mathbb{J} \kappa \in \sp(2d,\R)$ (see \eqref{eq:kappas}), we may define
\begin{equation*}
  \hat{\varkappa} \defeq
  \frac{1}{2}
  \begin{bmatrix}
    I_{2} \otimes \kappa_{12}^{T} & P \otimes \kappa_{22} \\
    -P \otimes \kappa_{11} & -I_{2} \otimes \kappa_{12}
  \end{bmatrix}
  = \frac{1}{2}
  \begin{bmatrix}
    \kappa_{12}^{T} & 0 & 0 & \kappa_{22} \\
    0 & \kappa_{12}^{T} & \kappa_{22} & 0 \\
    0 & -\kappa_{11} & -\kappa_{12} & 0 \\
    -\kappa_{11} & 0 & 0 & -\kappa_{12}
  \end{bmatrix}
  \in \sp(4d,\R).
\end{equation*}
Then $\hat{\varkappa} = \hat{\mathbb{J}} \hat{\kappa}$ with
\begin{equation}
  \label{eq:hatkappa}
  \hat{\kappa}
  \defeq
   \frac{1}{2}
  \begin{bmatrix}
    P \otimes \kappa_{11} & I_{2} \otimes \kappa_{12} \\
    I_{2} \otimes \kappa_{12}^{T} & P \otimes \kappa_{22}
  \end{bmatrix}
  = \frac{1}{2}
  \begin{bmatrix}
    0 & \kappa_{11} & \kappa_{12} & 0 \\
    \kappa_{11} & 0 & 0 & \kappa_{12} \\
    \kappa_{12}^{T} & 0 & 0 & \kappa_{22} \\
    0 & \kappa_{12}^{T} & \kappa_{22} & 0
  \end{bmatrix}
  \in \sym(4d,\R).
\end{equation}
Accordingly, we may define the infinitesimal generator in the extended phase space as
\begin{equation}
  \label{eq:inf_gen-ext}
  \hat{\kappa}_{T^{*}\R^{2d}}(\zeta)
  = \hat{\varkappa} \zeta
  = \hat{\mathbb{J}} \hat{\kappa} \zeta
  = \frac{1}{2}
  \begin{bmatrix}
    \kappa_{12}^{T} q + \kappa_{22} y \\
    \kappa_{12}^{T} x + \kappa_{22} p \\
    -\kappa_{11} x - \kappa_{12} p \\
    -\kappa_{11} q - \kappa_{12} y
  \end{bmatrix}.
\end{equation}

\subsection{Symmetry and Quadratic Invariants of Extended System}
\label{ssec:quad_inv-ext}
Now we are ready to state how the extended Hamiltonian system~\eqref{eq:Pihajoki} inherits symmetry and quadratic invariants from the original one~\eqref{eq:Ham}:
\begin{proposition}[Inheritance of symmetry and quadratic invariants]
  \label{prop:inheritance}
  Let $H\colon T^{*}\R^{d} \to \R$ be a smooth Hamiltonian and $\hat{H}\colon T^{*}\R^{2d} \to \R$ be its associated extended Hamiltonian defined as in \eqref{eq:hatH}.
  Suppose that $\mathsf{G}$ is a matrix Lie subgroup of $\Sp(2d,\R)$, and let $\mathfrak{g} \subset \sp(2d,\R)$ be its Lie algebra, and $\mathfrak{g}_{\sym} \defeq \mathbb{J}^{T}\mathfrak{g} \subset \sym(2d,\R)$.
  \begin{enumerate}[\bf (i)]
  \item If $H$ has $\mathsf{G}$-symmetry in the sense of \eqref{eq:G-symmetry}, then $\hat{H}$ inherits $\mathsf{G}$-symmetry via the action $\hat{\Psi}$ defined in \eqref{eq:hatPsi}:
    \begin{equation*}
      \hat{H} \circ \hat{\Psi}_{S}(\zeta) = \hat{H}(\zeta)
      \quad
      \forall
      S \in \mathsf{G}
      \quad
      \forall \zeta \in T^{*}\R^{2d}.
    \end{equation*}
    As a result, for any $\kappa \in \mathfrak{g}_{\sym}$, the quadratic function
    \begin{equation}
      Q_{\kappa}(z) \defeq \frac{1}{2} z^{T} \kappa z
      = \frac{1}{2} q^{T} \kappa_{11} q + q^{T} \kappa_{12} p + \frac{1}{2} p^{T} \kappa_{22} p
      \tag{\ref{eq:Q}}
    \end{equation}
    is an invariant of the original Hamiltonian system~\eqref{eq:Ham}, and
    \begin{equation}
      \label{eq:hatQ}
      \hat{Q}_{\kappa}(\zeta) \defeq \frac{1}{2} \zeta^{T} \hat{\kappa} \zeta
      = \frac{1}{2} \parentheses{ q^{T} \kappa_{11} x + q^{T} \kappa_{12} p + x^{T} \kappa_{12} y + y^{T} \kappa_{22} p },
    \end{equation}
    is an invariant of the extended Hamiltonian system~\eqref{eq:Pihajoki}.
    \label{prop:inheritance-i}
    \smallskip
  \item For any $\kappa \in \sym(2d,\R)$, \eqref{eq:Q} is a quadratic invariant of the original Hamiltonian system~\eqref{eq:Ham} if and only if \eqref{eq:hatQ} is a quadratic invariant of the extended Hamiltonian system~\eqref{eq:Pihajoki}.
    \label{prop:inheritance-ii}
  \end{enumerate}
\end{proposition}
\begin{remark}
  Every quadratic function on $T^{*}\R^{d}$ may be written as in \eqref{eq:Q} for an appropriate $\kappa \in \sym(2d,\R)$, whereas not every quadratic function on $T^{*}\R^{2d}$ may be written as in \eqref{eq:hatQ}.
\end{remark}
\begin{remark}
  The extended system of \citet{Ta2016b} (see \eqref{eq:Tao}) does \textit{not} inherit quadratic invariants in general; see \Cref{ssec:inheritance-Tao}.
\end{remark}
\begin{proof}[Proof of \Cref{prop:inheritance}]
  \leavevmode
  \begin{enumerate}[\bf (i)]
  \item The $\mathsf{G}$-symmetry of $\hat{H}$ follows from a straightforward computation:
    For any $S =
    \begin{bmatrix}
      \mathsf{A} & \mathsf{B} \\
      \mathsf{C} & \mathsf{D}
    \end{bmatrix}
    \in \mathsf{G}$ and any $\zeta = (q, x, p, y) \in T^{*}\R^{2d}$,
    \begin{align*}
      \hat{H} \circ \hat{\Psi}_{S}(q, x, p, y)
      &= H( \mathsf{A} q + \mathsf{B} y, \mathsf{C} q + \mathsf{D} y )
        + H( \mathsf{A} x + \mathsf{B} p,  \mathsf{C} x + \mathsf{D} p ) \\
      &= H\parentheses{ S
      \begin{bmatrix}
        q \\
        y
      \end{bmatrix}
      }
      + H\parentheses{ S
        \begin{bmatrix}
          x \\
          p
        \end{bmatrix}
      } \\
      &= H(q, y) + H(x, p) \\
      &= \hat{H}(q, x, p, y),
    \end{align*}
    where the third equality follows from \eqref{eq:G-symmetry}.

    As discussed in \Cref{ssec:quad_inv}, the $\mathsf{G}$-symmetry of $H$ implies its $\mathfrak{g}$-symmetry, and it in turn implies that \eqref{eq:Q} is an invariant of the original Hamiltonian system~\eqref{eq:Ham}.
    Similarly, the $\mathsf{G}$-symmetry of $\hat{H}$ implies the following $\mathfrak{g}$-symmetry of $\hat{H}$:
    \begin{equation}
      \label{eq:g-symmetry-ext}
      \hat{\kappa}_{T^{*}\R^{2d}}(\zeta)^{T} D\hat{H}(\zeta) = 0
      \quad
      \forall
      \kappa \in \mathfrak{g}_{\sym}
      \quad
      \forall \zeta \in T^{*}\R^{2d}.
    \end{equation}
    However, in view of \eqref{eq:inf_gen-ext} and \eqref{eq:hatQ}, we have
    \begin{equation*}
      \hat{\kappa}_{T^{*}\R^{2d}}(\zeta) = \hat{\mathbb{J}}\,D\hat{Q}_{\hat{\kappa}}(\zeta)
      \quad
      \forall
      \zeta \in T^{*}\R^{2d},
    \end{equation*}
    and so we have, for any $\zeta \in T^{*}\R^{2d}$,
    \begin{align*}
      \{\hat{Q}_{\kappa}, \hat{H}\}_{\text{ext}}(\zeta)
      &= D\hat{Q}_{\kappa}(\zeta)^{T} \hat{\mathbb{J}}\, D\hat{H}(\zeta) \\
      &= -\hat{\kappa}_{T^{*}\R^{2d}}(\zeta)^{T} D\hat{H}(\zeta) \\
      &= 0.
    \end{align*}
    Hence $\hat{Q}_{\kappa}$ is an invariant of the extended system~\eqref{eq:Pihajoki}.
  \item Let $\kappa \in \sym(2d,\R)$ be arbitrary.
    Recall from \Cref{ssec:quad_inv} that
    \begin{equation}
      \label{eq:g-symmetry2}
      \begin{split}
        \text{\eqref{eq:Q} is an invariant of \eqref{eq:Ham}}
        &\iff
          \{Q_{\kappa}, H\}(z) = 0\quad \forall z \in T^{*}\R^{d}
          \\
        &\iff \kappa_{T^{*}\R^{d}}(z)^{T} DH(z) = 0 \quad \forall z \in T^{*}\R^{d}.
      \end{split}
    \end{equation}
    On the other hand, using \eqref{eq:inf_gen-ext}, we also have
    \begin{equation}
      \label{eq:g-symmetry-ext2}
      \begin{split}
        \text{\eqref{eq:hatQ} is an invariant of \eqref{eq:Pihajoki}}
        &\iff
        \{\hat{Q}_{\kappa}, \hat{H}\}_{\text{ext}}(\zeta) = 0 \quad \forall \zeta \in T^{*}\R^{2d}
        \\
        &\iff
        \hat{\kappa}_{T^{*}\R^{2d}}(\zeta)^{T} D\hat{H}(\zeta) = 0
        \quad \forall \zeta \in T^{*}\R^{2d}.
      \end{split}
    \end{equation}
    However, notice that we have the following equality:
    \begin{equation}
      \label{eq:g-symmetry2-equiv}
    \begin{split}
      \hat{\kappa}_{T^{*}\R^{2d}}(\zeta)^{T} D\hat{H}(\zeta)
      &= \frac{1}{2}
        \begin{bmatrix}
          \kappa_{12}^{T} q + \kappa_{22} y \\
          \kappa_{12}^{T} x + \kappa_{22} p \\
          -\kappa_{11} x - \kappa_{12} p \\
          -\kappa_{11} q - \kappa_{12} y
        \end{bmatrix}
      \cdot
      \begin{bmatrix}
        D_{1}H(q,y) \\
        D_{1}H(x,p) \\
        D_{2}H(x,p) \\
        D_{2}H(q,y)
      \end{bmatrix}
      \\
      &= \frac{1}{2} \kappa_{T^{*}\R^{d}}(q,y)^{T} DH(q,y)
        + \frac{1}{2} \kappa_{T^{*}\R^{d}}(x,p)^{T} DH(x,p).
    \end{split}
    \end{equation}
    Suppose that the left-hand side vanishes for any $\zeta \in T^{*}\R^{2d}$.
    Then, setting $(x,p) = 0$ yields
    \begin{equation*}
      \kappa_{T^{*}\R^{d}}(q,y)^{T} DH(q,y) = 0
      \quad \forall (q,y) \in T^{*}\R^{d},
    \end{equation*}
    which is clearly equivalent to \eqref{eq:g-symmetry2}.
    Conversely, if \eqref{eq:g-symmetry2} holds then it clearly implies \eqref{eq:g-symmetry-ext2} in view of \eqref{eq:g-symmetry2-equiv}.
    Therefore, \eqref{eq:g-symmetry2} and \eqref{eq:g-symmetry-ext2} are equivalent.
  \end{enumerate}
\end{proof}

\begin{example}[Quadratic invariant of extended NLS system]
  As we have seen in \Cref{ex:NLS2}, the NLS possesses the quadratic invariant $Q_{\kappa}$ shown in \eqref{eq:Q-NLS} with $\kappa = 2I_{2d}$.
  Hence we have
  \begin{equation*}
    \hat{\kappa} = 
    \begin{bmatrix}
      0 & I_{d} & 0 & 0 \\
      I_{d} & 0 & 0 & 0 \\
      0 & 0 & 0 & I_{d} \\
      0 & 0 & I_{d} & 0
    \end{bmatrix},
  \end{equation*}
  and so the corresponding extended system possesses the quadratic invariant
  \begin{equation*}
    \hat{Q}_{\kappa}(\zeta)
    = \frac{1}{2} \zeta^{T} \hat{\kappa} \zeta
    = q^{T} x + y^{T} p
    = \eta^{T} \xi.
  \end{equation*}
  Notice that, while the original invariant $Q_{\kappa}(z) = \sum_{i=1}^{d} (q_{i}^{2} + p_{i}^{2})$ had no ``mixed term'' like $q^{T}p$, the invariant $\hat{Q}_{\kappa}(\zeta)$ for the extended system consists only of the mixed term $\eta^{T} \xi$.
  We shall see in the next section that this generalizes to any quadratic invariant of \eqref{eq:Ham} and is one of the key observations towards the proof of our main result.
\end{example}

\section{Conservation Laws in Extended Phase Space Integrators}
\label{sec:proof}
\subsection{Pihajoki's Integrator}
\label{ssec:Pihajoki}
Recall from \Cref{sec:extended_integrators} that, writing $\eta = (q, y)$ and $\xi = (x, p)$, we may write the extended system~\eqref{eq:Pihajoki} as follows:
\begin{equation}
  \dot{\eta} = \mathbb{J}\,DH(\xi),
  \qquad
  \dot{\xi} = \mathbb{J}\,DH(\eta),
  \tag{\ref{eq:Pihajoki-xi_eta}}
\end{equation}
and that Pihajoki's integrator~\eqref{eq:Pihajoki-Strang} is the extended-phase-space analogue of the St\"ormer--Verlet method.
This implies the following lemma on the invariants inherited by Pihajoki's integrator:

\begin{lemma}[Linear and quadratic invariants of Pihajoki's integrator]
  \label{lem:Pihajoki}
  Let $\dt > 0$ and $\hat{\zeta_{0}} \in T^{*}\R^{2d}$ be arbitrary and set $\hat{\zeta_{1}} \defeq \hat{\Phi}_{\dt}(\hat{\zeta}_{0})$, where $\hat{\Phi}$ is Pihajoki's integrator~\eqref{eq:Pihajoki-Strang}.
  Then:
  \begin{enumerate}[\bf (i)]
  \item If the Hamiltonian system~\eqref{eq:Ham} possesses a linear invariant of the form~\eqref{eq:L} in $T^{*}\R^{d}$ with $a \in \R^{2d}$, then Pihajoki's extended phase space integrator~\eqref{eq:Pihajoki-Strang} preserves the linear invariant of the form~\eqref{eq:hatL} in $T^{*}\R^{2d}$, i.e.,
    \begin{equation*}
      \hat{L}_{a}(\hat{\zeta}_{0}) = \hat{L}_{a}(\hat{\zeta}_{1}).
    \end{equation*}
  \item If the Hamiltonian system~\eqref{eq:Ham} possesses a quadratic invariant of the form~\eqref{eq:Q} in $T^{*}\R^{d}$ with $\kappa \in \sym(2d,\R)$, then Pihajoki's extended phase space integrator~\eqref{eq:Pihajoki-Strang} preserves a quadratic invariant of the form~\eqref{eq:hatQ} in $T^{*}\R^{2d}$, i.e.,
    \begin{equation*}
      \hat{Q}_{\kappa}(\hat{\zeta}_{0}) = \hat{Q}_{\kappa}(\hat{\zeta}_{1}).
    \end{equation*}
  \end{enumerate}
\end{lemma}
\begin{remark}
  \label{rem:Pihajoki}
  Unfortunately, this lemma does \textit{not} imply that Pihajoki's integrator~\eqref{eq:Pihajoki-Strang} preserves linear and quadratic invariant of the \textit{original} Hamiltonian system~\eqref{eq:Ham} in $T^{*}\R^{d}$.
  In other words, it is a conservation law that holds only in the \textit{extended} phase space $T^{*}\R^{2d}$.
  We shall discuss this issue in \Cref{ssec:discussion} below.
\end{remark}
\begin{proof}[Proof of \Cref{lem:Pihajoki}]
  \leavevmode
  \begin{enumerate}[\bf (i)]
  \item By the assumption and \Cref{prop:inheritance-linear}, the linear function \eqref{eq:hatL} is an invariant of the extended Hamiltonian system~\eqref{eq:Pihajoki}.
    The St\"ormer--Verlet-type splitting~\eqref{eq:Pihajoki-Strang} is the partitioned Runge--Kutta method with the 2-stage Lobatto IIIA--IIIB pair applied to \eqref{eq:Pihajoki-xi_eta}, and is known to preserve any linear invariant of the system; see~\citet[Section~II.2.1 and Theorems~IV.1.5]{HaLuWa2006}.
  \item By the assumption and \Cref{prop:inheritance}~\eqref{prop:inheritance-ii}, the quadratic function \eqref{eq:hatQ} is an invariant of the extended Hamiltonian system~\eqref{eq:Pihajoki}.
    Now, notice that we may rewrite \eqref{eq:hatQ} as follows:
    \begin{align*}
      \hat{Q}_{\kappa}(\zeta)
      &= \frac{1}{2} \parentheses{ q^{T} \kappa_{11} x
        + q^{T} \kappa_{12} p
        + y^{T} \kappa_{12}^{T} x
        + y^{T} \kappa_{22} p
        }
      \\
    &= \frac{1}{2} \begin{bmatrix}
        q^{T} & y^{T}
      \end{bmatrix}
                \begin{bmatrix}
                  \kappa_{11} & \kappa_{12} \\
                  \kappa_{12}^{T} & \kappa_{22}
                \end{bmatrix}
                                    \begin{bmatrix}
                                      x \\
                                      p
                                    \end{bmatrix}
      \\
      &= \frac{1}{2} \eta^{T} \kappa \xi.
    \end{align*}
    
    It is well known that the St\"ormer--Verlet-type splitting~\eqref{eq:Pihajoki-Strang} for systems of the form~\eqref{eq:Pihajoki-xi_eta} preserves quadratic invariants of the form $\eta^{T} M \xi$ with $M \in \R^{2d \times 2d}$; see, e.g., \cite[Section~II.2 and Theorems~IV.2.3]{HaLuWa2006}.
  \end{enumerate}
\end{proof}

\subsection{Semiexplicit Integrator: Proof of \Cref{thm:main}}
\label{ssec:proof}
Let $z_{0} = (q_{0}, p_{0}) \in T^{*}\R^{d}$ be arbitrary and $\dt > 0$ be chosen such that $z_{1} \defeq \Phi_{\dt}(z_{0})$ is defined, where $\Phi$ is the discrete flow of the semiexplicit integrator defined in \Cref{ssec:semiexplicit}.

Before getting into the details of the proof, let us recall from \Cref{ssec:semiexplicit} how the semiexplicit method works.
Given $z_{0} = (q_{0}, p_{0})$, set $\zeta_{0} = (q_{0}, q_{0}, p_{0}, p_{0}) \in \mathcal{N}$ (i.e., $(x_{0},y_{0}) = (q_{0},p_{0})$), and the symmetric projection determines $\mu \in \R^{2d}$ so that
\begin{equation}
  \label{eq:zeta_1}
  \zeta_{1} \defeq \hat{\zeta}_{1} + A^{T} \mu \in \mathcal{N} = \ker A,
\end{equation}
where
\begin{equation}
  \label{eq:zeta_0}
  \hat{\zeta}_{1} \defeq \hat{\Phi}_{\dt}(\hat{\zeta}_{0})
  \quad\text{with}\quad
  \hat{\zeta}_{0} \defeq \zeta_{0} + A^{T} \mu
\end{equation}
using Pihajoki's integrator $\hat{\Phi}$ from \eqref{eq:Pihajoki-Strang}.
This means that $\zeta_{1} = (q_{1}, x_{1}, p_{1}, y_{1})$ satisfies $(x_{1}, y_{1}) = (q_{1}, p_{1})$, and thus one sets $z_{1} = (q_{1}, p_{1}) = \Phi_{\dt}(z_{0})$.
We also note that one can write $\mu$ in terms of $\hat{\zeta}_{0}$ or $\hat{\zeta}_{1}$: Since $\zeta_{0}, \zeta_{1} \in \ker A$ and $A A^{T} = 2 I_{2d}$ (see \eqref{eq:A}), we see from \eqref{eq:zeta_1} and \eqref{eq:zeta_0} that
\begin{equation}
  \label{eq:mu-zeta}
  \mu = -\frac{1}{2} A \hat{\zeta_{1}}
  = \frac{1}{2} A \hat{\zeta_{0}}.
\end{equation}

Suppose that the original Hamiltonian system~\eqref{eq:Ham} possesses a linear invariant of the form
\begin{equation*}
  L_{a}(z) \defeq a^{T} z
\end{equation*}
with $a \in \R^{2d}$ as well as a quadratic invariant of the form
\begin{equation}
  Q_{\kappa}(z) \defeq \frac{1}{2} z^{T} \kappa z
  = \frac{1}{2} q^{T} \kappa_{11} q + q^{T} \kappa_{12} p + \frac{1}{2} p^{T} \kappa_{22} p
  \tag{\ref{eq:Q}}
\end{equation}
with $\kappa \in \sym(2d,\R)$.
We would like to prove that $L_{a}(z_{0}) = L_{a}(z_{1})$ and $Q_{\kappa}(z_{0}) = Q_{\kappa}(z_{1})$.
It suffices to show that
\begin{equation*}
  L_{a}(z_{1}) - L_{a}(z_{0}) = \hat{L}_{a}(\hat{\zeta}_{1}) - \hat{L}_{a}(\hat{\zeta}_{0})
  \quad\text{and}\quad
  Q_{\kappa}(z_{1}) - Q_{\kappa}(z_{0}) = \hat{Q}_{\kappa}(\hat{\zeta}_{1}) - \hat{Q}_{\kappa}(\hat{\zeta}_{0}),
\end{equation*}
because \Cref{lem:Pihajoki} says that the right-hand side of each of these equations vanishes.

\subsubsection{Linear case}
First observe that
\begin{equation*}
  L_{a}(z_{1}) - L_{a}(z_{0})
  = a^{T} (z_{1} - z_{0})
  = \hat{a}^{T} (\zeta_{1} - \zeta_{0})
\end{equation*}
using the definition~\eqref{eq:hatL} of $\hat{a}$ and also $(x_{i}, y_{i}) = (q_{i}, p_{i})$ for $\zeta_{i} = (q_{i}, x_{i}, p_{i}, y_{i})$ for $i = 0, 1$.

On the other hand,
\begin{align*}
  \hat{L}_{a}(\hat{\zeta}_{1}) - \hat{L}_{a}(\hat{\zeta}_{0})
  &= \hat{a}^{T} (\hat{\zeta}_{1} - \hat{\zeta}_{0}) \\
  &= \hat{a}^{T} (\zeta_{1} - \zeta_{0}) - 2 \hat{a}^{T} A^{T} \mu \\
  &= \hat{a}^{T} (\zeta_{1} - \zeta_{0}),
\end{align*}
where we used \eqref{eq:zeta_1} and \eqref{eq:zeta_0} for the second equality; the last equality follows because $\hat{a} \in \ker A$; see \eqref{eq:A} and \eqref{eq:hatL}.
Hence we have
\begin{equation*}
  L_{a}(z_{1}) - L_{a}(z_{0}) = \hat{L}_{a}(\hat{\zeta}_{1}) - \hat{L}_{a}(\hat{\zeta}_{0}).
\end{equation*}

\subsubsection{Quadratic case}
The key observation is the following: Defining
\begin{equation*}
  \bar{\kappa} \defeq \frac{1}{2}
  \begin{bmatrix}
    I_{2} \otimes \kappa_{11} & P \otimes \kappa_{12} \\
    P \otimes \kappa_{12}^{T} & I_{2} \otimes \kappa_{22}
  \end{bmatrix} \\
  = \frac{1}{2}
  \begin{bmatrix}
    \kappa_{11} & 0 & 0 & \kappa_{12} \\
    0 & \kappa_{11} & \kappa_{12} & 0 \\
    0 & \kappa_{12}^{T} & \kappa_{22} & 0 \\
    \kappa_{12}^{T} & 0 & 0 & \kappa_{22}
  \end{bmatrix}
  \in \sym(4d,\R)
\end{equation*}
and 
\begin{equation*}
  \bar{Q}_{\kappa}(\zeta) \defeq \frac{1}{2} \zeta^{T} \bar{\kappa} \zeta
  = \frac{1}{4} \parentheses{ q^{T} \kappa_{11} q + x^{T} \kappa_{11} x }
  + \frac{1}{2} \parentheses{ q^{T} \kappa_{12} y + x^{T} \kappa_{12} p }
  + \frac{1}{4} \parentheses{ p^{T} \kappa_{22} p + y^{T} \kappa_{22} y },
\end{equation*}
we have
\begin{equation*}
  Q_{\kappa}(z_{i}) = \bar{Q}_{\kappa}(\zeta_{i})
  \quad\text{for}\quad
  i = 0, 1
\end{equation*}
because $\zeta_{i} = (q_{i}, x_{i}, p_{i}, y_{i}) \in \mathcal{N}$, i.e., $(x_{i}, y_{i}) = (q_{i}, p_{i})$ for $i = 0, 1$.
Hence it suffices to show that
\begin{equation*}
  \bar{Q}_{\kappa}(\zeta_{1}) - \bar{Q}_{\kappa}(\zeta_{0}) = \hat{Q}_{\kappa}(\hat{\zeta}_{1}) - \hat{Q}_{\kappa}(\hat{\zeta}_{0}). 
\end{equation*}

To that end, observe that, using \eqref{eq:zeta_1} and \eqref{eq:zeta_0}, 
\begin{align*}
  \bar{Q}_{\kappa}(\zeta_{1}) - \bar{Q}_{\kappa}(\zeta_{0})
  &= \bar{Q}_{\kappa}\parentheses{ \hat{\zeta}_{1} + A^{T} \mu } - \bar{Q}_{\kappa}\parentheses{ \hat{\zeta}_{0} - A^{T} \mu } \\
  &= \frac{1}{2} \hat{\zeta}_{1}^{T} \bar{\kappa} \hat{\zeta}_{1} + \hat{\zeta}_{1}^{T} \bar{\kappa} A^{T} \mu + \frac{1}{2} \mu^{T} A \bar{\kappa} A^{T} \mu \\
  &\quad- \frac{1}{2} \hat{\zeta}_{0}^{T} \bar{\kappa} \hat{\zeta}_{0} + \hat{\zeta}_{0}^{T} \bar{\kappa} A^{T} \mu - \frac{1}{2} \mu^{T} A \bar{\kappa} A^{T} \mu \\
  &= \frac{1}{2} \hat{\zeta}_{1}^{T} \bar{\kappa} \hat{\zeta}_{1} + \hat{\zeta}_{1}^{T} \bar{\kappa} A^{T} \mu
  - \frac{1}{2} \hat{\zeta}_{0}^{T} \bar{\kappa} \hat{\zeta}_{0} + \hat{\zeta}_{0}^{T} \bar{\kappa} A^{T} \mu \\
  &= \frac{1}{2} \hat{\zeta}_{1}^{T} \bar{\kappa} \parentheses{I_{4d}  - A^{T} A} \hat{\zeta}_{1}
  - \frac{1}{2} \hat{\zeta}_{0}^{T} \bar{\kappa} \parentheses{I_{4d}  - A^{T} A} \hat{\zeta}_{0},
\end{align*}
where we used \eqref{eq:mu-zeta} for the last equality.

Now, using the definition~\eqref{eq:A} of $A$, we see
\begin{equation*}
  I_{4d} - A^{T} A =
  I_{2d} -
  \begin{bmatrix}
    I_{d} & -I_{d} & 0 & 0 \\
    -I_{d} & I_{d} & 0 & 0 \\
    0 & 0 & I_{d} & -I_{d} \\
    0 & 0 & -I_{d} & I_{d}
  \end{bmatrix}
  =
  \begin{bmatrix}
    P \otimes I_{d} & 0 \\
    0 & P \otimes I_{d}
  \end{bmatrix},
\end{equation*}
and so, noting that $P^{2} = I_{2}$,
\begin{align*}
  \bar{\kappa}\parentheses{ I_{4d} - A^{T} A }
  &= \frac{1}{2}
    \begin{bmatrix}
      I_{2} \otimes \kappa_{11} & P \otimes \kappa_{12} \\
      P \otimes \kappa_{12}^{T} & I_{2} \otimes \kappa_{22}
    \end{bmatrix}
                                      \begin{bmatrix}
                                        P \otimes I_{d} & 0 \\
                                        0 & P \otimes I_{d}
                                      \end{bmatrix}
  \\
  &= \frac{1}{2}
    \begin{bmatrix}
      P \otimes \kappa_{11} & I_{2} \otimes \kappa_{12} \\
      I_{2} \otimes \kappa_{12}^{T} & P \otimes \kappa_{22}
    \end{bmatrix} \\
  &= \hat{\kappa}
\end{align*}
in view of \eqref{eq:hatkappa}.
Therefore, we have
\begin{equation*}
  \bar{Q}_{\kappa}(\zeta_{1}) - \bar{Q}_{\kappa}(\zeta_{0}) =
  \frac{1}{2} \hat{\zeta}_{1}^{T} \hat{\kappa} \hat{\zeta}_{1}
  - \frac{1}{2} \hat{\zeta}_{0}^{T} \hat{\kappa} \hat{\zeta}_{0}
  = \hat{Q}_{\kappa}(\hat{\zeta}_{1}) - \hat{Q}_{\kappa}(\hat{\zeta}_{0}).
\end{equation*}
This completes the proof of \Cref{thm:main}.

\section{Discussion and Numerical Results}
\label{sec:numerical_results}
\subsection{Discussion: Conservation and Non-conservation}
\label{ssec:discussion}
As we have mentioned in \Cref{rem:Pihajoki}, \Cref{lem:Pihajoki} does \textit{not} imply that Pihajoki's \textit{integrator}~\eqref{eq:Pihajoki-Strang} preserves linear and quadratic invariant of the \textit{original} Hamiltonian system~\eqref{eq:Ham} in $T^{*}\R^{d}$.
This is because the existence of an invariant of the integrator~\eqref{eq:Pihajoki-Strang} in the \textit{extended} phase space $T^{*}\R^{2d}$ does \textit{not} imply the existence of an invariant in the \textit{original} phase space $T^{*}\R^{d}$.
More specifically, note that $\hat{\zeta}_{1} = (\hat{q}_{1}, \hat{x}_{1}, \hat{p}_{1}, \hat{y}_{1})$ does \textit{not} satisfy $(\hat{x}_{1}, \hat{y}_{1}) = (\hat{q}_{1}, \hat{p}_{1})$ in general even if $\hat{\zeta}_{0} = (\hat{q}_{0}, \hat{x}_{0}, \hat{p}_{0}, \hat{y}_{0})$ satisfies $(\hat{x}_{0}, \hat{y}_{0}) = (\hat{q}_{0}, \hat{p}_{0})$.
Therefore, even if $\hat{L}_{a}(\hat{\zeta}_{0})$ is written in terms of $(\hat{q}_{0}, \hat{p}_{0})$ and is an invariant of the integrator~\eqref{eq:Pihajoki-Strang}, one has $\hat{L}_{a}(\hat{\zeta}_{1})$ in terms of $(\hat{q}_{1}, \hat{x}_{1}, \hat{p}_{1}, \hat{y}_{1})$ with $(\hat{x}_{1}, \hat{y}_{1}) \neq (\hat{q}_{1}, \hat{p}_{1})$ in the next step.
Hence it is impossible to interpret $\hat{L}_{a}$ as an invariant on the \textit{original} phase space $T^{*}\R^{d}$ in terms of $(q,p)$.
The same goes with the quadratic invariant $\hat{Q}_{\kappa}$.

Tao's integrator~\eqref{eq:Tao-Strang} also has the same issue along with an additional issue for quadratic invariants: Tao's extended system~\eqref{eq:Tao} lacks the inheritance of quadratic invariants; see \Cref{ssec:inheritance-Tao}.
So even if the original Hamiltonian system~\eqref{eq:Ham} possesses a quadratic invariant in $T^{*}\R^{d}$, Tao's extended system~\eqref{eq:Tao} may not have a corresponding invariant even in the extended phase space $T^{*}\R^{2d}$ due to the additional step (see \eqref{eq:Tao-Strang}) added to prevent the defect from growing, as we shall discuss in \Cref{sec:Tao}.

One can also see these issues more explicitly in terms of the defect
\begin{equation*}
  \delta{z}_{n} \defeq (\delta{q}_{n}, \delta{p}_{n}) \defeq (\hat{x}_{n} - \hat{q}_{n}, \hat{y}_{n} - \hat{p}_{n})
\end{equation*}
at the $n^{\rm th}$ step of numerical solution.
For a linear invariant, \Cref{lem:Pihajoki} implies that, for any $n \in \N$,
\begin{equation*}
  \hat{L}_{a}(\hat{\zeta}_{n}) = \hat{L}_{a}(\hat{\zeta}_{0}) \defeq \ell_{0}
  \implies
  L_{a}(\hat{q}_{n}, \hat{p}_{n}) = \ell_{0} + \frac{1}{2} L_{a}(\delta{z}_{n}).
\end{equation*}
Hence the deviation of the original invariant $L_{a}$ from the constant value $\ell_{0}$ is proportional to the defect.
Since the defect $\delta{z}_{n}$ often tends to grow for Pihajoki's integrator, $L_{a}$ may also grow as well.
For Tao's integrator, $\delta{z}_{n}$ tends to oscillate without drift, and so $L_{z}(z_{n})$ also oscillates in a similar way as we shall see in a moment.
We shall numerically demonstrate these issues in the next subsection.

Interestingly, however, Pihajoki's integrator preserves those linear invariants in terms of $q$ or $p$ only, i.e., of the form $a_{q}^{T} q$ and $a_{p}^{T} p$ with $a_{q}, a_{p} \in \R^{d}$.
This follows from a straightforward calculation based on the definition of the integrator.
On the other hand, Tao's integrator does not possess the same property again due to the additional step.

\Cref{tab:comparison} gives a summary of which integrator preserves what types of invariants exactly.

\begin{table}[htbp]
  \renewcommand{\arraystretch}{1.2}
  \centering
  \caption{
    Preservation/non-preservation of linear and quadratic invariants by three extended phase space integrators, where $z = (q, p) \in \R^{2d}$, $a \in \R^{2d}$, $a_{q}, a_{p} \in \R^{d}$, and $\kappa \in \sym(2d,\R)$.
    The check mark (\cmark) indicates that the integrator preserves the invariant of the original Hamiltonian system~\eqref{eq:Ham} of the given form exactly in general, whereas the cross mark (\xmark) indicates that the integrator does not preserve it exactly in general.
  }
  \begin{tabular}{c|c|c|c|}
    & \multicolumn{3}{c|}{Invariant} \\\cline{2-4}
    & $a^{T} z$ & $a_{q}^{T} q$ or $a_{p}^{T} p$ & $z^{T}\kappa z/2$ \\\hline
    \citet{Pi2015} & \xmark & \cmark & \xmark \\\hline
    \citet{Ta2016b} & \xmark & \xmark & \xmark \\\hline
    Semiexplicit~\cite{JaOh2023} & \cmark & \cmark & \cmark \\\hline
  \end{tabular}
  \label{tab:comparison}
\end{table}

\subsection{Numerical Results}
In order to numerically demonstrate the results in \Cref{tab:comparison}, let us first devise a simple test case that possesses both linear and quadratic invariants:
\begin{example}[Test case with $d = 2$]
  \label{ex:test_case}
  Consider the Hamiltonian system~\eqref{eq:Ham} with the following non-separable Hamiltonian on $T^{*}\R^{2}$:
  \begin{gather*}
    H(q, p) = \exp( f(q_{1}, p_{1}) ) \sin( g(q_{2}, p_{2}) ), \\
    f(x,y) \defeq \frac{1}{10} (2x - 3y),
    \qquad
    g(x,y) \defeq \frac{1}{4}( x^{2} + 2y^{2} ),
  \end{gather*}
  where $q = (q_{1}, q_{2}), p = (p_{1}, p_{2}) \in \R^{2}$; the subscripts stand for components not the time steps here.
  
  It is straightforward to show that $L(q,p) \defeq f(q_{1}, p_{1})$ is a linear invariant of the system; this implies that $Q(q,p) \defeq g(q_{2},p_{2})$ is a quadratic invariant of the system because the Hamiltonian $H(q,p)$ is an invariant.
  As a result, the trajectories are very simple: a straight line $f(q_{1}, p_{1}) = \text{const.}$ on the $q_{1}$-$p_{1}$ plane and an ellipse $g(q_{2}, p_{2}) = \text{const.}$ on the $q_{2}$-$p_{2}$ plane.

  \Cref{fig:test_case} shows the time evolutions of the norm $\norm{ (x - q, y - p) }$ of the defect and the relative errors of the invariants $L$ and $Q$ using Pihajoki's, Tao's, and our semiexplicit integrators with the initial condition $q(0) = (-1, 2)$ and $p(0) = (1, -1)$ and the time step $\dt = 0.1$.
  \begin{figure}[htbp]
    \centering
    \includegraphics[width=0.9\linewidth]{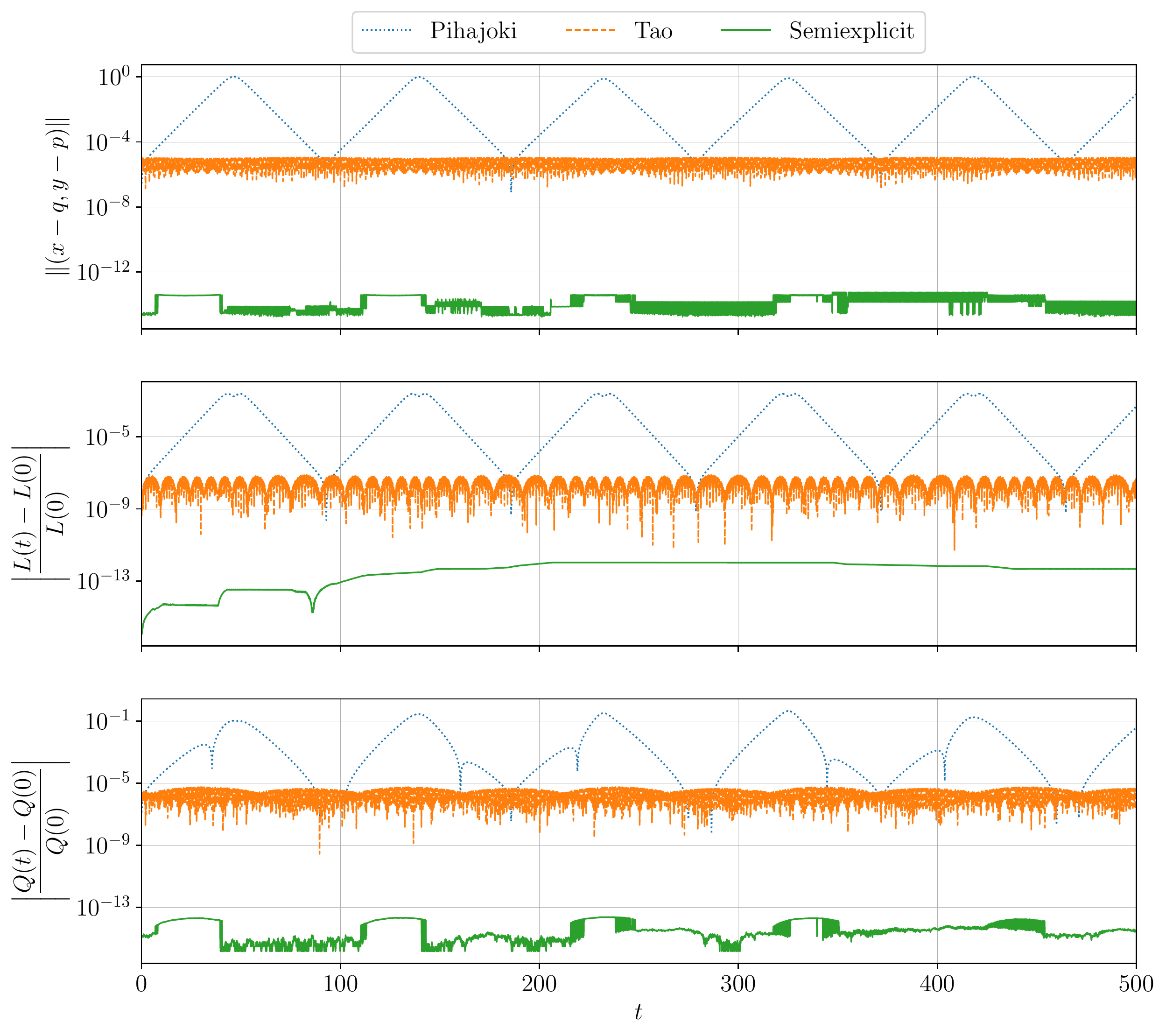}
    \caption{
      Time evolutions of norm $\norm{ (x - q, y - p) }$ of defect and relative errors of linear and quadratic invariants $L$ and $P$ for test case in \Cref{ex:test_case} with $\dt = 0.1$.
      Tao's integrator uses $\omega = 10$, and the tolerance for the nonlinear solver in the semiexplicit integrator is $\epsilon = 10^{-14}$.
      We defined $L(t) \defeq L(q(t),p(t))$ and similarly for $Q$.
      Those points that give 0 for vertical values are removed from the plots.
    }
    \label{fig:test_case}
  \end{figure}
  
  We observe that, despite the simplicity of the solution behavior, Pihajoki's and Tao's integrators preserve neither of the linear and quadratic invariants $L$ and $Q$ exactly.
  On the other hand, our semiexplicit integrator preserves both invariants roughly up to the tolerance $\epsilon = 10^{-14}$ used in the nonlinear solver for $\mu$ (the simplified Newton method discussed in \cite[Section~4.1]{JaOh2023}).
  This demonstrates the exact preservation stated in \Cref{thm:main}.
  One also observes that, for Pihajoki's and Tao's integrators, the fluctuation of the invariants is roughly proportional to the norm of the defect $(x - q, y - p)$ as discussed above.
\end{example}

Let us next consider a more practical example that possesses both linear and quadratic invariants:
\begin{example}[Point vortices]
  \label{ex:vortices}
  We consider the dynamics of $N$ point vortices in $\mathbb{R}^{2}$ with circulations $\{ \Gamma_{i} \in \R\backslash\{0\} \}_{i=1}^{N}$.
  The motion of the centers $\{ \mathbf{x}_{i} = (\mathsf{x}_{i},\mathsf{y}_{i}) \in \R^{2} \}_{i=1}^{N}$ of the vortices is governed by
  \begin{equation*}
    \dot{\mathsf{x}}_{i} = -\frac{1}{2\pi} \sum_{\substack{1\le j \le N\\ j \neq i}} \Gamma_{j} \frac{\mathsf{y}_{i} - \mathsf{y}_{j}}{\norm{\mathbf{x}_{i} - \mathbf{x}_{j}}^{2}},
    \qquad
    \dot{\mathsf{y}}_{i} = \frac{1}{2\pi} \sum_{\substack{1\le j \le N\\ j \neq i}} \Gamma_{j} \frac{\mathsf{x}_{i} - \mathsf{x}_{j}}{\norm{\mathbf{x}_{i} - \mathbf{x}_{j}}^{2}}
  \end{equation*}
  for $i \in \{1, \dots, N\}$; see, e.g., \citet[Section~2.1]{Ne2001} and \citet[Section~2.1]{ChMa1993}.
  It is known to be a Hamiltonian system with the Hamiltonian
  \begin{equation*}
    H(\mathbf{x}_{1}, \dots, \mathbf{x}_{N})
    \defeq -\frac{1}{4\pi} \sum_{1\le i < j \le N} \Gamma_{i} \Gamma_{j} \ln\norm{\mathbf{x}_{i} - \mathbf{x}_{j}}^{2},
  \end{equation*}
  but not in the canonical sense.
  However, one may rewrite the system in the canonical form~\eqref{eq:Ham} upon the change of coordinates
  \begin{equation*}
    (\mathsf{x}_{i}, \mathsf{y}_{i}) \mapsto \parentheses{ \sqrt{|\Gamma_{i}|}\, \mathsf{x}_{i}, \sqrt{|\Gamma_{i}|}\, \sgn(\Gamma_{i})\, \mathsf{y}_{i} } \eqdef (q_{i}, p_{i}),
  \end{equation*}
  where $\sgn(x)$ is $1$ if $x > 0$ and $-1$ otherwise.
  So we have $d = N$ here, and the Hamiltonian is again non-separable.

  This system has three invariants in addition to the Hamiltonian:
  \begin{equation}
    \label{eq:invariants-vortices}
    \begin{array}{c}
      \DS L_{a}(z) \defeq \sum_{i=1}^{N} \Gamma_{i} \mathsf{x}_{i} = \sum_{i=1}^{N} \sqrt{|\Gamma_{i}|}\, \sgn(\Gamma_{i})\, q_{i},
      \qquad
      L_{b}(z) \defeq \sum_{i=1}^{N} \Gamma_{i} \mathsf{y}_{i} = \sum_{i=1}^{N} \sqrt{|\Gamma_{i}|}\, p_{i},
      \\
      \DS Q_{\kappa}(z) \defeq \sum_{i=1}^{N} \Gamma_{i} \norm{ \mathbf{x}_{i} }^{2} = \sum_{i=1}^{N} \sgn(\Gamma_{i}) \parentheses{ q_{i}^{2} + p_{i}^{2} }
    \end{array}
  \end{equation}
  with
  \begin{align*}
    a &\defeq \parentheses{ \sqrt{|\Gamma_{1}|}\, \sgn(\Gamma_{1}), \dots, \sqrt{|\Gamma_{N}|}\, \sgn(\Gamma_{N}), 0, \dots, 0 } \in \R^{2N}, \\
    b &\defeq \parentheses{ 0, \dots, 0, \sqrt{|\Gamma_{1}|}, \dots, \sqrt{|\Gamma_{N}|} } \in \R^{2N}, \\
    \kappa &\defeq 2
             \begin{bmatrix}
               \sigma & 0 \\
               0 & \sigma
             \end{bmatrix}
                   \in \sym(2N,\R)
                   \quad\text{with}\quad
                   \sigma \defeq \diag\parentheses{ \sgn(\Gamma_{1}), \dots, \sgn(\Gamma_{N}) }.
  \end{align*}
  The pair $(L_{a},L_{b})$ is called the linear impulse, and $Q_{\kappa}$ is called the angular impulse.

  We consider the case with four vortices ($N = d = 4$) with circulations
  \begin{equation*}
    (\Gamma_{1}, \Gamma_{2}, \Gamma_{3}, \Gamma_{4}) = (4, -3, -2, 7).
  \end{equation*}
  \Cref{fig:test_case} shows the time evolutions of the norm $\norm{ (x - q, y - p) }$ of the defect and the relative errors of the linear and quadratic invariants $(L_{a}, L_{b})$ and $Q_{\kappa}$ using Pihajoki's, Tao's, and our semiexplicit integrators with the initial positions of the vortices at
  \begin{equation*}
    \bigl\{ (\mathsf{x}_{i}(0), \mathsf{y}_{i}(0)) \bigr\}_{i=1}^{4} = \bigl\{ (1, 2),\, (-3/2, 1),\, (-3, -1),\, (2, 1/2) \bigr\},
  \end{equation*}
  and with time step $\dt = 0.05$; the tolerance for the nonlinear solver in the semiexplicit integrator is again $\epsilon = 10^{-14}$.
  \begin{figure}[htbp]
    \centering
    \includegraphics[width=0.9\linewidth]{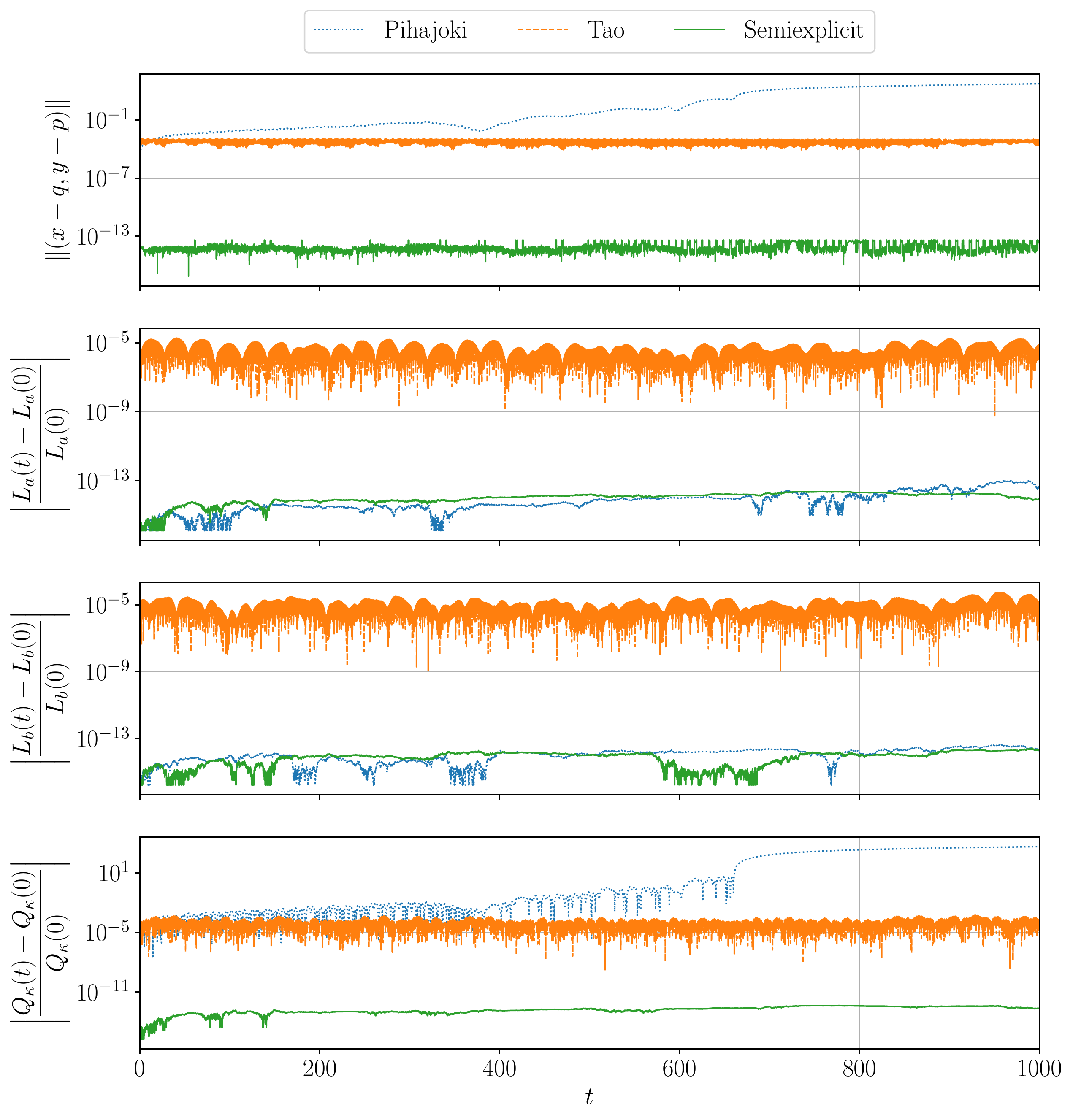}
    \caption{
      Time evolutions of norm $\norm{ (x - q, y - p) }$ of defect and relative errors of linear and quadratic invariants $(L_{a}, L_{b})$ and $Q_{\kappa}$ for the 4-vortex problem from \Cref{ex:vortices} with $\dt = 0.05$.
      The rest of the details is the same as in \Cref{fig:test_case}.
    }
    \label{fig:4vortices}
  \end{figure}

  We observe that Pihajoki's integrator preserves the linear invariants $(L_{a}, L_{b})$ almost exactly despite a problematic growth of the defect.
  This is because these linear invariants are the kind that is exactly preserved by Pihajoki's integrator; see \Cref{tab:comparison} and \eqref{eq:invariants-vortices}.
  On the other hand, the relative error in the quadratic invariant $Q_{\kappa}$ is growing even to the scale of 10 to 100.
  One also sees that the growth is roughly proportional to the defect, again as discussed in \Cref{ssec:discussion}.
  
  Tao's integrator exhibits good preservation of the linear and quadratic invariants.
  However, note that Tao's integrator does not preserve any of the invariants exactly as shown in \Cref{tab:comparison}.
  Indeed, one again sees that the errors are roughly proportional to the defect as discussed in \Cref{ssec:discussion}.
  
  On the other hand, our semiexplicit integrator preserves all three linear and quadratic invariants roughly up to the tolerance $\epsilon = 10^{-14}$.
  This result again demonstrates the exact preservation stated in \Cref{thm:main} and \Cref{tab:comparison}.
\end{example}

\section*{Acknowledgments}
I would like to thank Buddhika Jayawardana for helpful discussions, Robert McLachlan for insightful comments, and the editor and the reviewers for their criticisms and comments.
This work was supported by NSF grant DMS-2006736.

\appendix

\section{Numerical Results on Efficiency}
\label{sec:efficiency}
\subsection{Semiexplicit vs.~Gauss--Legendre}
As pointed out by one of the reviewers of the present paper, our implementation of the Gauss--Legendre (GL) methods in our previous work~\cite{JaOh2023} using full Newton's method was very inefficient, and resulted in inflating the computational costs for the GL methods.

Following a suggestion from the reviewer, we implemented the GL methods using fixed point iterations instead, and performed a numerical study on the computational costs using the same examples from \cite{JaOh2023}.
The results, as we shall show in the subsections to follow, suggest that the semiexplicit integrator and the GL methods are comparable in computational efficiency.
While the 2nd- and 4th- order GL methods are faster than the semiexplicit ones of the same orders, the 6th-order method of the latter can be faster than the former of the same order.

The computational cost of these methods is a trade-off between the number of evaluations of the vector field per single iteration and the number of iterations.
In general, the semiexplicit integrator requires more evaluations of the vector field compared to the GL method of the same order \textit{per single iteration} for solving the nonlinear equation.
However, the vector $\mu$ (see \Cref{ssec:semiexplicit}) in the semiexplicit method is typically very small especially for higher-order methods, because $\mu$ is a quantity that vanishes for the exact solutions.
On the other hand, the GL methods need to solve for unknowns of $O(\dt)$ in general.
Therefore, as we shall see in the results to follow, the semiexplicit method usually requires fewer iterations especially with higher-order methods, while the GL methods tend to require more or less the same number of iterations for all orders.

Another reason why the semiexplicit method can compensate for the disadvantage with higher-order methods is the following:
While the semiexplicit method solves for $\mu \in \R^{2d}$ regardless of the order and the number of stages, the GL methods with $s$ stages ($s = 2, 3$ for the 4th- and 6th-order methods) solves for unknowns in $\R^{2sd}$.

\subsection{Finite-Dimensional NLS}
As the first test case, we consider the finite-dimensional NLS from \Cref{ex:NLS}.
Following \citet{Ta2016b}, we have $d = 5$, $\omega = 100$, and $q(0) = (3,0.01,0.01,0.01,0.01)$ and $p(0) = (1,0,0,0,0)$.

See \Cref{tab:NLS} for the results.
The 4th-order and 6th-order versions (\textsf{Tao\,4} and \textsf{Tao-Y\,6}) of Tao's method use the Triple Jump composition (see \cite{CrGo1989,Forest1989,Su1990,Yo1990}; also \cite[Example~II.4.2]{HaLuWa2006}) and the composition of \citet{Yo1990}, respectively.
The same goes with the 4th- and 6th-order semiexplicit methods (\textsf{semiexplicit\,4} and \textsf{semiexplicit-Y\,6}).
The \textsf{GL\,$n$} stands for the $n$-th order Gauss--Legendre method; see, e.g., \cite[Section~II.1.3]{HaLuWa2006} and \cite[Table~6.4 on p.~154]{LeRe2004} implemented with fixed point iterations~\cite[Section~VIII.6]{HaLuWa2006}.

\begin{table}[htbp]
  \caption{
    Comparison of computation times of various methods when solving NLS~\eqref{eq:H-NLS} with time step $\dt = 10^{-3}$ and terminal time $T=10^{3}$.
    \texttt{Time} is the computation time in seconds averaged over 5 simulations.
    \texttt{Itr} is the average number of iterations used per step in the simplified Newton and fixed point iterations for the semiexplicit and the GL methods, respectively.
    \texttt{VF\_eval} is the average number of evaluations of the vector field.
    Computations were performed using Julia on a computer with an Apple M1 Pro processor.
  }
  \centering
  \begin{tabular}{l|rcc||rcc}
    & \multicolumn{3}{c||}{$\epsilon=10^{-10}$} & \multicolumn{3}{c}{$\epsilon=10^{-13}$} \\ \hline
    \textsf{Method} & \multicolumn{1}{c}{\texttt{Time}} & \texttt{Itr} & \texttt{VF\_eval} & \multicolumn{1}{c}{\texttt{Time}} & \texttt{Itr} & \texttt{VF\_eval} \\ \hline
    \textsf{Tao\,2}	       & 9.23  & 	&  4    &  9.23	&  	& 4 \\
    \textsf{semiexplicit\,2}   & 23.00 & 3.385  & 10.16 & 32.47 & 5.000 & 15.00\\
    \textsf{GL\,2}             & 15.03 & 5.999	& 5.999 & 17.31	& 7.000	& 7.000 \\ \hline
    \textsf{Tao\,4}	       & 22.99 & 	& 12    & 22.99	&   	& 12 \\
    \textsf{semiexplicit\,4}   & 31.96 & 1.948	& 17.53 & 47.47 & 3.000 & 27 \\
    \textsf{GL\,4}	       & 26.08 & 5.000	& 10.00 & 31.35	& 6.320	& 12.64 \\ \hline
    \textsf{Tao-Y\,6}	       & 50.86 & 	& 28    & 50.86	& 	& 28 \\
    \textsf{semiexplicit-Y\,6} & 35.29 & 1.002	& 21.04 & 35.39	& 1.017	& 21.36 \\
    \textsf{GL\,6}	       & 39.64 & 5.000	& 15.00 & 48.50	& 6.234 & 18.70 \\ \hline
  \end{tabular}
  \label{tab:NLS}
\end{table}

\subsection{Point Vortices}
As the second test case, consider the vortex dynamics from \Cref{ex:vortices} with 10 vortices ($N=10$) of circulations
\begin{equation}
  \label{eq:Gammas}
  (\Gamma_{1}, \dots, \Gamma_{10}) = \frac{1}{10}(-5,3,6,7,-2,-8,-9,-3,7,-6).
\end{equation}
and the initial condition
\begin{equation}
  \label{eq:IC-10vortex}
  \begin{split}
    \bigl\{ (\mathsf{x}_{i}(0), \mathsf{y}_{i}(0)) \bigr\}_{i=1}^{10}
    = \bigl\{&
    (3, -5),\,
    (-10, -6),\,
    (6, 0),\,
    (9, -2),\,
    (0, 0),\, \\
    &(7, 10),\,
    (-8, 2),\,
    (5, 9),\,
    (9, 0),\,
    (7, -1)
    \bigr\}.
  \end{split}
\end{equation}
See \Cref{tab:10vortices} for the results.

\begin{table}[htbp]
  \caption{
    Comparison of computation times of various methods when solving the 10-vortex system with the parameters given in \eqref{eq:Gammas} and \eqref{eq:IC-10vortex}, and time step $\dt = 0.1$ and terminal time $T=10^{3}$.
    The other details are the same as \Cref{tab:NLS} except $\omega = 7$ for Tao's method.
  }
  \centering
  \begin{tabular}{l|rcc||rcc}
   & \multicolumn{3}{c||}{$\epsilon=10^{-10}$} & \multicolumn{3}{c}{$\epsilon=10^{-13}$} \\ \hline
    \textsf{Method} & \multicolumn{1}{c}{\texttt{Time}} & \texttt{Itr} & \texttt{VF\_eval} & \multicolumn{1}{c}{\texttt{Time}} & \texttt{Itr} & \texttt{VF\_eval} \\ \hline
    \textsf{Tao\,2}	       & 5.66   & 	& 4     & 5.66  &  	 & 4 \\
    \textsf{semiexplicit\,2}   & 8.43	& 2.014 & 6.042 & 13.26 & 3.045  & 9.135 \\
    \textsf{GL\,2}             & 5.63	& 4.023	& 4.023 & 8.65  & 5.997  & 5.997 \\ \hline
    \textsf{Tao\,4}	       & 17.02  & 	& 12    & 17.02 &        & 12 \\
    \textsf{semiexplicit\,4}   & 12.52	& 1.001	& 9.009 & 20.47 & 1.588  & 14.29 \\
    \textsf{GL\,4}	       & 10.99	& 4.000	& 8.000 & 14.29 & 5.000  & 10.00 \\ \hline
    \textsf{Tao-Y\,6}	       & 39.23	& 	& 28    & 39.23 & 	 & 28 \\
    \textsf{semiexplicit-Y\,6} & 28.76	& 1.000	& 21.00 & 30.06 & 1.001  & 21.02 \\
    \textsf{GL\,6}	       & 16.47	& 4.000	& 12.00 & 21.45 & 5.000  & 15.00 \\ \hline
  \end{tabular}
  \label{tab:10vortices}
\end{table}

\section{Limitation of Inheritance by Tao's Extended System}
\label{sec:Tao}
\subsection{Tao's Extended Phase Space Integrator}
In order to suppress the defect $(x - q, y - p)$ that often grows with Pihajoki's integrator~\eqref{eq:Pihajoki-Strang}, \citet{Ta2016b} proposed to solve the extended system
\begin{equation}
  \label{eq:Tao}
  \begin{aligned}
    \dot{q} &= D_{2}H(x, p) + \omega(p - y), \qquad & \dot{p} &= -D_{1}H(q, y) - \omega(q - x), \medskip\\
    \dot{x} &= D_{2}H(q, y) + \omega(y - p), \qquad  & \dot{y} &= -D_{1}H(x, p) - \omega(x - q)
  \end{aligned}
\end{equation}
with some $\omega \in \R \backslash\{ 0 \}$ instead of Pihajoki's extended system~\eqref{eq:Pihajoki}.
Note that the above system~\eqref{eq:Tao} is also a Hamiltonian system on the extended phase space $T^{*}\R^{2d}$ with the Hamiltonian
\begin{equation}
  \label{eq:hatH-Tao}
  \hat{H}_{\text{T}}(\zeta) \defeq \hat{H}(\zeta) + \hat{H}_{C}(\zeta)
  \quad\text{with}\quad
  \hat{H}_{C}(\zeta) \defeq \frac{\omega}{2}\parentheses{ (x - q)^{2} + (y - p)^{2} },
\end{equation}
where $\hat{H}$ is the extended Hamiltonian~\eqref{eq:hatH} for Pihajoki's system~\eqref{eq:Pihajoki}.
The Strang splitting~\cite{St1968} then yields the following $2^{\rm nd}$-order integrator:
\begin{equation}
  \label{eq:Tao-Strang}
  \hat{\Phi}^{A}_{\dt/2} \circ \hat{\Phi}^{B}_{\dt/2} \circ \hat{\Phi}^{C}_{\dt} \circ \hat{\Phi}^{B}_{\dt/2} \circ \hat{\Phi}^{A}_{\dt/2},
\end{equation}
where $\hat{\Phi}^{A}$ and $\hat{\Phi}^{B}$ are defined in \eqref{eq:PhiAB} and $\Phi^{C}$ is the (extended) Hamiltonian flow corresponding to $\hat{H}_{C}$.

\subsection{Inheritance and Non-inheritance of Linear and Quadratic Invariants}
\label{ssec:inheritance-Tao}
It turns out that, while the extended system~\eqref{eq:Tao} enjoys a similar inheritance property as in \Cref{prop:inheritance-linear} for linear invariants, it does not for quadratic ones as in \Cref{prop:inheritance}.

To see this for linear invariant, note that
\begin{align*}
  \{\hat{L}_{a}, \hat{H}_{C}\}_{\text{ext}}(\zeta) = \hat{a}^{T} \hat{\mathbb{J}}\,D\hat{H}_{C}(\zeta) = 0
  \quad \forall \zeta \in T^{*}\R^{2d}.
\end{align*}
Hence the additional term $\hat{H}_{C}$ in the Hamiltonian does not interfere with the preservation of $\hat{L}_{a}$ by Tao's extended \textit{system}~\eqref{eq:Tao}.
Note however that this does \textit{not} imply that Tao's \textit{integrator}~\eqref{eq:Tao-Strang} preserves the original linear invariant $L_{a}$ for the same reason discussed in \Cref{ssec:discussion} for Pihajoki's integrator.

On the other hand, for quadratic invariants,
\begin{align*}
  \{\hat{Q}_{\kappa}, \hat{H}_{C}\}_{\text{ext}}(\zeta) 
  &= \hat{\kappa}_{T^{*}\R^{2d}}(\zeta)^{T} D\hat{H}_{C}(\zeta) \\
  &= \delta{z}^{T}
    \begin{bmatrix}
      \kappa_{12} & -\kappa_{11} \\
      -\kappa_{22} & \kappa_{12}
    \end{bmatrix}
                     \delta{z}
                     \quad\text{with}\quad
                     \delta{z} \defeq
                     \begin{bmatrix}
                       x - q \\
                       y - p
                     \end{bmatrix}.
\end{align*}
Hence we have
\begin{equation*}
  \{\hat{Q}_{\kappa}, \hat{H}_{C}\}_{\text{ext}}(\zeta) = 0
  \quad \forall \zeta \in T^{*}\R^{2d}
  \iff
  \left\{
    \begin{aligned}
      \kappa_{12}^{T} = -\kappa_{12}, \\
      \kappa_{22} = -\kappa_{11}.
    \end{aligned}
  \right.
\end{equation*}
Therefore, a quadratic invariant $\hat{Q}_{\kappa}$ of Pihajoki's extended system~\eqref{eq:Pihajoki} is also an invariant of Tao's extended system~\eqref{eq:Tao} if and only if $\kappa \in \sym(2d,\R)$ takes the form
\begin{equation*}
  \kappa =
  \begin{bmatrix}
    \kappa_{11} & \kappa_{12} \\
    -\kappa_{12} & -\kappa_{11}
  \end{bmatrix}
  \quad\text{with}\quad
  \kappa_{11} \in \sym(d,\R),
  \
  \kappa_{12}^{T} = -\kappa_{12}.
\end{equation*}
However, this is rather restrictive.
Indeed, none of the quadratic invariants from \Cref{ex:NLS2,ex:test_case,ex:vortices} satisfy this condition.

\bibliographystyle{plainnat}
\bibliography{Semiexplicit2}

\end{document}